\documentclass[12pt]{article}

\usepackage{authblk}

\usepackage{amsmath,amsthm,amsfonts,amssymb,mathbbol}
\usepackage[scr]{rsfso}

\usepackage{tikz-cd}

\usepackage[export]{adjustbox}

\usepackage[a4paper,
            centering,
            left=25mm,
            right=25mm,
            top=30mm,
            bottom=30mm
            ]{geometry}

\usepackage[unicode,
			colorlinks,
			linkcolor=orange,
			citecolor=blue,
			urlcolor=magenta]{hyperref}

\theoremstyle{plain}
\newtheorem{theorem}{Theorem}[section]
\newtheorem{proposition}[theorem]{Proposition}
\newtheorem{lemma}[theorem]{Lemma}
\newtheorem{corollary}[theorem]{Corollary}

\theoremstyle{definition}
\newtheorem{definition}[theorem]{Definition}
\newtheorem{example}[theorem]{Example}

\theoremstyle{remark}
\newtheorem{remark}[theorem]{Remark}

\newenvironment{prf}{{\noindent \textbf{Proof:}\ }}{\hfill $\Box$ \\ \smallskip}
\numberwithin{equation}{section}

\title{Additivity of higher rho invariant for topological structure group from a differential point of view}
\author[1]{Baojie Jiang}
\author[2 $\dagger$]{Hongzhi Liu}
\affil[1]{College of Mathematics and Statistics, Chongqing University,\protect\\ Chongqing 401331, P. R. China.\protect\\ e-Mail: jiangbaojie@gmail.com}
\affil[2]{School of Mathematics, Shanghai University of Finance and Economics,\protect\\ Shanghai 200433, P. R. China.\protect\\ e-Mail: liu.hongzhi@mail.shufe.edu.cn}
\affil[$\dagger$]{Corresponding author}
\date{\today}

\begin{document}
\maketitle

\abstract{
In \cite{WXY16}, Weinberger, Xie and Yu proved that higher rho invariant associated to homotopy equivalence defines a group homomorphism from the topological structure group to analytic structure group, $K$-theory of certain geometric $C^*$-algebras, by piecewise-linear approach. In this paper, we adapt part of  Weinberger, Xie and Yu's work, to give a differential geometry theoretic proof of the additivity of the map induced by higher rho invariant associated to homotopy equivalence on topological structure group.

\medskip



\section{Introduction}

In 2005, Higson and Roe (cf: \cite{HR1,HR2, HR3}) constructed a transformation from smooth surgery exact sequence to the $K$-theory exact sequence of geometric $C^*$-algebras (also called the analytic surgery exact sequence).
In particular, higher rho invariant associated to smooth homotopy equivalence between two smooth manifolds induces a set theoretic map from smooth strucuture set to $K$-theory of certain $C^*$-algebras. To their end, Higson and Roe developed two approaches to the construction of higher rho invariant: the piecewise-linear approach, giving rise to theory of signature of bounded Hilbert-Poincar\'e complex, and the differential geometric approach, giving rise to theory of signature of unbounded Hilbert-Poincar\'e complex.

Later, Piazza and Schick gave an index theoretic construction of higher rho invariant associated to homotopy equivalence (cf: \cite{PS2014a, PS2016a}).
In \cite{Z17}, Zenobi extended the work of Higson and Roe, Piazza and Schick to the case of topological manifolds and showed that the higher rho invariant is a well-defined set theoretic map form the structure group of topological manifolds to $K$-theory of certain $C^*$-algebras.
 
Unlike smooth structure set, the topological structure set carries a group structure.
At the time, it was an open question whether the higher rho invariant induces group homomorphism on topological structure group.
In the breakthrough work of Weinberger, Xie and Yu (\cite{WXY16}), the question was answered positively in complete generality.
There are two major novelties of their work:
\begin{enumerate}
	\item[(1)] It gives a new description of the topological structure group in terms of smooth manifolds with boundary.
More specifically, the new description of structure group allows one to replace topological manifolds in the usual definition of structure group (of topological manifolds) by PL or smooth manifolds with boundary.
Such a description leads to a transparent group structure given by disjoint union;
	\item[(2)] It develops a theory of higher rho invariants in this new setting, in which higher rho invariants are easily seen to be additive. The additivity of higher rho invariants allows them to give an estimation of nonrigidity of topological manifold.
\end{enumerate}
In \cite{WXY16}, the theory of higher rho invariants is built up in terms of the PL version of signature operators, or in other words, bounded Hilbert-Poincar\'e complexes.
In this paper, we adapt this part of \cite{WXY16} to work on differential geometric version of signature operators, i.e., unbounded Hilbert-Poincar\'e complexes.
Although it seems to be straightforward in principle, it is quite nontrivial to carry out all the technical details. Besides, a differential geometric approach to Weinberger, Xie and Yu's theory benefits us at the computation of higher rho invariant on topological structure group. 

This paper is organized as follows.
In section \ref{section Preliminary}, we collect notations and basic facts on geometric $C^*$-algebras.
In section \ref{homotopy invariance of higher signature}, we recall Signature operators and homotopy invariance of higher signature.
Our treatment in section \ref{homotopy invariance of higher signature} rely on a detailed understanding of \cite{HS92,W2013a}.
In section \ref{subsection Define for element of L} and section \ref{Define for element of N} we examine index map for $L_n(\pi_1 X)$ and local index map for $\mathcal{N}_n(X)$ respectively at great length.
Roughly speaking, under our construction, homotopy equivalence represents zero element in $L_n(\pi_1 X)$, while the \emph{infinitesimal controled homotopy equivalence} represents zero element in $\mathcal{N}_n(X)$ (\cite{WXY16}).
It should be point out that the index map for $L_n(\pi_1 X)$ is well-known if one uses the classical definition of $L$-groups.
The key point here is that we using Wall’s geometric definition of $L$-groups (cf: \cite{W1970a}), which is given in terms of manifolds with boundary (with additional data).
Therefore, a new construction of the index map is needed.
In section \ref{section Mapping surgery to analysis}, we define higher rho map on topological structure group and prove its additivity. Together With the aid of the results of section \ref{subsection Define for element of L} and section \ref{Define for element of N}, for closed oriented topological manifolds of dimension $n\geq 5$, we also construct the mapping surgery to analysis.
For the convenience of readers, we recall the reinterpretation of structure group, normal group, $L$-group given by Weinberger, Xie, and Yu in \cite{WXY16} in Appendix \ref{section WXY surgery sequence}, and Poincar\'e duality operators in Appendix \ref{section Poin duality oper}.

Throughout this article, manifolds are suppose to be oriented.
All maps between manifolds are suppose to be oriented-preserving.
We will use $G$ denote a countable discrete group unless otherwise specified.

\medskip

\noindent{\textit{Acknowledgements.}} Thanks Yu Guoliang for suggesting us this question and all those necessary guidance.
We also want to thank Xie Zhizhang for his help.
We would like also to thank Shanghai Center for Mathematical Sciences and Department of Mathematics of Texas A\&M University for their hospitality. 
The research is partially supported by the Fundamental Research Funds for the Central Universities (2019CDXYST0015) and NNSF of China (11771092).

\section{Preliminary}\label{section Preliminary}
In this section, we collect some basic notations and terminologies needed in this paper.
We refer the reader to \cite{R1996a,Y97} for more details.

Let $X$ be a \emph{proper} metric space (a metric space is called \emph{proper} if every closed ball is compact).
An \emph{$X$-module} is a separable Hilbert space equipped with a $*$-representation $\pi$ of $C_0(X)$, where $C_0(X)$ is the $C^*$-algebra of all complex-valued continuous functions on $X$ which vanish at infinity.
An $X$-module is called \emph{nondegenerate} if the $*$-representation $\pi$ of $C_0(X)$ is nondegenerate.
An $X$-module is said to be \emph{standard} if no nonzero function in $C_0(X)$ acts as a compact operator.
When $H_X$ is an $X$-module, for each $f\in C_0(X)$ and $\xi\in H_X$, we denote $\pi(f)(\xi)$ by $f\xi$.

\begin{definition}\label{def propa local pseudo}
Let $H_X$ be an $X$-module, $T:H_X\to H_X$ a bounded linear operator.
\begin{enumerate}
\item The \emph{propagation} of $T$ is defined to be $\sup\{d(x,y) ~|~ (x,y)\in \text{Supp}(T) \},$ where $\text{Supp}(T)$ is the complement (in $X\times X$) of the set of points $(x,y) \in X\times X$ for which there exists $f,g\in C_0(X)$ such that $gTf=0$ and $f(x)g(y)\neq 0$;
\item $T$ is said to be \emph{locally compact} if the operators $fT$ and $Tf$ are compact for all $f\in C_0(X)$.
\end{enumerate}
\end{definition}

\begin{definition}\label{def roe and localization}
Let $H_X$ be a standard nondegenerate $X$-module and $B(H_X)$ the set of all bounded linear operators on $H_X$.
\begin{enumerate}
\item The \emph{$Roe$-algebra} of $X$, denoted by $C^*(X)$, is the $C^*$-algebra generated by all locally compact operators in $B(H_X)$ with finite propagation ;
\item The \emph{Localization algebra} $C_L^*(X)$ is the norm completion of the algebra of all bounded and uniformly-norm continuous functions $f : [1,\infty)\to C^*(X)$ such that 	
\[
\text{propagation of }f(t)\to 0\  \text{~as~}  t \to \infty,
\]
with respect to the norm $\|f\|_\infty := \sup_{t\geq 1}\|f(t)\|$.
\item The \emph{obstruction algebra} $C^*_{L,0}(X)$ is defined to be the kernel of the \emph{evaluation map}
\begin{center}
\begin{tabular}{rccc}
$ev$:&$C_L^*(X)$&$\to$ &$C^*(X)$\\
&$f$&$\mapsto$&$f(1)$.\\
\end{tabular}	
\end{center}
In particular, $C^*_{L,0}(X)$ is an ideal of $C_L^*(X)$;

\item If $Y$ is a subspace of $X$, then $C^*_L(Y;X)$ is defined to be a closed subalgebra of $C_L^*(X)$  generated by all elements $f$ such that there exist $c_t>0$ satisfying $\lim_{t\to \infty} c_t=0$, and $\text{Supp}(f(t))\subset \{(x,y)\in X\times X~|~ d((x,y), Y\times Y)\leq c_t\ \}$ for all $t$.
Similarly, we can define $C^*_{L,0}(Y;X)$ which is a closed subalgebra of $C^*_{L,0}(X)$.
\end{enumerate}
\end{definition}

Let $G$ be a (countable) discrete group and acts properly on $X$ by isometries.
Let $H_X$ be a $X$-module, the $*$-representation of $C_0(X)$ denoted by $\pi$.
Let $u:G\to B(H_X)$ be a unitary representation of $G$ which is covariant in the sense that
\[
u_g \pi(f)=\pi(g.f)u_g
\]
where, $g.f(x)= f(g^{-1}x)$ for any $f\in C_0(X)$, $g \in G$.
Such a triple $(H_X, G, \pi,u)$ is called a \emph{covariant system}.

\begin{definition}
A covariant system $(H_X, G, \pi,u)$ is called \emph{admissible} if
\begin{enumerate}
\item The action of $G$ is \emph{proper} and \emph{cocompact};
\item $H_X$ is a nondegenerate standard $X$-module;
\item For each $x\in X$, the \emph{stabilizer group} $G_x$ acts on $H_X$ regularly in the sense that the action is isomorphic to the obvious action of $G_x$ on $l^2(G_x)\otimes H$ for some infinite dimensional Hilbert space $H$.
Here $G_x$ acts on $\ell^2 (G_x)$ by (left) translations and acts on $H$ trivially.
\end{enumerate}
\end{definition}

We remark that, for any locally compact metric space $X$ with a proper and cocompact isometric action of $G$, such an admissible covariant system $(H_X, G, \pi,u)$ always exists.
We will simply denote an admissible covariant system $(H_X, G, \pi,u)$ by $H_X$ and call it an \emph{admissible} $(X, G)$-module.

\begin{definition}
Let $X$ be a locally compact metric space $X$ with a proper and cocompact isometric action of $G$.
Let $H_X$ be an admissible $(X, G)$-module, denoted by $\mathbb{C}[X]^G$ the $*$-algebra of all $G$-\emph{invariant} locally compact operators in $B(H_X)$ with finite propagation.
The $G$-\emph{equivariant Roe algebra} of $X$, denoted by $C^*(X)^G$, is defined to be the operator norm closure of $\mathbb{C}[X]^G$ in $B(H_X)$.
Similarly, we can define $C_L^*(X)^G$, $C_{L,0}^*(X)^G$, $C_{L}^*(Y;X)^G$ and $C_{L,0}^*(Y;X)^G$.
\end{definition}

\begin{remark}
Up to isomorphism, $C^*(X)$ does not depend on the choice of the standard nondegenerate $X$-module $H_X$.
The same holds for $C_L^*(X)$, $C_{L,0}^*(X)$, $C_L^*(Y;X)$, $C_{L,0}^*(Y;X) $ and their $G$-\emph{equivariant} versions.
\end{remark}

\section{Homotopy invariance of higher signature}\label{homotopy invariance of higher signature}
The purpose of this section is to review the literature on the homotopy invariance of the higher signatures for closed manifolds (\cite{HR1,HR2,HS92}).
We begin with the definition of signature operator, which is consistent with the one in \cite[Section 5]{HR2}.

Let $M$ be an oriented, closed smooth (complete Riemannian) manifold of dimension $n$ and let $ \tilde{M}$ be a regular $G$-covering space of $M$, where $G$ is a discrete group.

Use the notation of Example \ref{ex.de Rham complex}, write $\Lambda^p_{L^2}(\tilde{M})$ for the square integrable differential $p$-forms on $\tilde{M}$.
We denote by $d_{\tilde{M}}$ the operator-closure of the \emph{de Rham differential operator} on $\Lambda^p_{L^2}(\tilde{M})$.
Let
\[
S_{\tilde{M}}(\omega)=i^{p(p-1)+[\frac{n}{2}]}*\omega,\;\omega\in \Lambda^p(\tilde{M}),
\]
which is self-adjoint with $S_{\tilde{M}}^2=1$, anti-commutes with the operator $d_{\tilde{M}} + d^*_{\tilde{M}}$.
And $S_{\tilde{M}}$ a \emph{Poincar\'e duality operator} of $(\bigoplus \Lambda^p_{L^2}(\tilde{M}),d_{\tilde{M}})$.
For notation simplicity, we will use $L^2(\Lambda(\tilde{M}))$ instead of $\bigoplus \Lambda^p_{L^2}(\tilde{M})$.



\begin{definition}[{\cite[Definition 5.4]{HR2}}]
If the dimension $n$ of $\tilde{M}$ is even then the signature operator $D_{\tilde{M}}$ on $\tilde{M}$ is given by $d_{\tilde{M}} + d^*_{\tilde{M}}$, viewed as an operator graded by $S_{\tilde{M}}$.
If the dimension $n$ of $\tilde{M}$ is odd then the signature operator $D_{\tilde{M}}$ on $\tilde{M}$ is given by the self-adjoint operator $i(d_{\tilde{M}}+d^*_{\tilde{M}})S_{\tilde{M}}$, restricting to even forms.
\end{definition}

The \emph{higher index} of the signature operator $D_{\tilde{M}}$ on $\tilde{M}$, denoted by $\text{Ind}(D_{\tilde{M}})$, is equal to the signature of the Hilbert-Poincar\'e complex $(L^2(\Lambda(\tilde{M})),d_{\tilde{M}}, S_{\tilde{M}})$ (\cite[Definition 5.10, Definition 5.11]{HR1}).
More specifically:
when $n$ is even,\;$K_0(C^*(\tilde{M})^G)$ is the receptacle for the higher index of the signature operator, and the higher index can be computed as
\[
[P_+(D_{\tilde{M}}+S_{\tilde{M}})]-[P_+(D_{\tilde{M}}-S_{\tilde{M}})] \in  K_0(C^*(\tilde{M})^G)
\]
where, $P_+(D_{\tilde{M}}+S_{\tilde{M}})$ and $P_+(D_{\tilde{M}}-S_{\tilde{M}})$ denote the positive projections of $D_{\tilde{M}}+S_{\tilde{M}}$ and $D_{\tilde{M}}-S_{\tilde{M}}$ respectively; when $n$ is odd,\;$K_1(C^*(\tilde{M})^G)$ is the receptacle for the higher index of the signature operator, and the higher index can be represented by the following invertible operator
\[
(D_{\tilde{M}}+S_{\tilde{M}}) (D_{\tilde{M}}-S_{\tilde{M}})^{-1}:\Lambda^{even} \to \Lambda^{even}
\]
where, $\Lambda^{even} =\bigoplus \Lambda^{2p}_{L^2}(\tilde{M})$ denote the even forms.
\begin{remark}
Actually, for $t>0$, we have	
\[
[P_+(D_{\tilde{M}}+tS_{\tilde{M}})]-[P_+(D_{\tilde{M}}-tS_{\tilde{M}})] = [P_+(D_{\tilde{M}}+S_{\tilde{M}})]-[P_+(D_{\tilde{M}}-S_{\tilde{M}})]\in K_0(C^*(\tilde{M})^G),
\]
and
\[
[(D_{\tilde{M}}+tS_{\tilde{M}})(D_{\tilde{M}}-tS_{\tilde{M}})^{-1}] = [(D_{\tilde{M}}+S_{\tilde{M}})(D_{\tilde{M}}-S_{\tilde{M}})^{-1}]\in K_1(C^*(\tilde{M})^G).
\]
\end{remark}


\subsection{A construction of bounded cochain map}\label{ssec.HSsubmersion}
In this subsection, we will recall some constructions exposed in \cite{HS92, PS07, W2013a},
using which we will describe the homotopy invariance of higher signature in subsection \ref{subsection Homotopy invariance of higher signature}.


Let $M$, $N$ be oriented, closed smooth (complete Riemannian) manifolds of dimension $n$ and let $f:M\to N$ be a smooth map.
Generally, the pull-back map $f^* : L^2(\Lambda(N)) \to L^2(\Lambda(M))$ does not induce a bounded operator from $L^2(\Lambda(N))$ to $L^2(\Lambda(M))$.
However, one can fix this problem by Hilsum-Skandalis submersion (\cite[Page 5]{HS92}).  
In fact, let $I =(-1,1)$ be an open interval.
For $k\in 8\mathbb{N}$ big enough, there exists a smooth map $p: I^k\times M \to N$ which is a \emph{submersion} such that $p(x,0)=f(x)$.
Let $v \in \Lambda^k (I^k )$ be a real-valued $k$-form with $\int_{I^k} v = 1$.
We define the map $T_v(p): L^2(\Lambda(N)) \to L^2(\Lambda(M))$ by
\[
T_f(p)(\omega)=\int_{I^k} v \wedge p^* \omega.
\]

%

Now, let $\tilde{N}$ be a regular $G$-covering of $N$.
Let $\tilde{M}=f^*(\tilde{N})$ be the pull back covering.
Then $\tilde{M}$ is a regular $G$-covering of $M$, and $f: M \to N$ can be lifted to a $G$-equivariant map $\tilde{f}: \tilde{M} \to \tilde{N}$. Similarly we can define $T_{\tilde{f}}$.

We list several properties of $T_{\tilde{f}}$ which will be used in this paper.
\begin{enumerate}
\item The operator $T_{\tilde{f}}:L^2(\Lambda(\tilde{N})) \to L^2(\Lambda(\tilde{M}))$ is uniformly bounded and $G$-equivariant;
\item We have $T_{\tilde{f}}(\text{dom}(d_{\tilde{N}}) ) \subset  \text{dom}(d_{\tilde{M} })$ and $d_{\tilde{M}}T_{\tilde{f}}=T_{\tilde{f}} d_{\tilde{N}}$, i.e. $T_{\tilde{f}}$ induces a chain map from $(L^2(\Lambda(\tilde{N})), d_{\tilde{N}})$ to $(L^2(\Lambda(\tilde{M})), d_{\tilde{M}})$;
\item If $f: M \to N$ is an orientation-preserving homotopy equivalence, then $T_{\tilde{f}}$ induces an isomorphism from $\text{Ker}(d_{\tilde{N}}) / \text{Im}(d_{\tilde{N}})$ to $\text{Ker} (d_{\tilde{M}})/ \text{Im}(d_{\tilde{M}})$ and inverse is $T_{\tilde{g}}$, where $g:N\to M$ is a homotopy inverse of $f$;
\item Let $v\in (L^2(\Lambda(\tilde{M})), d_{\tilde{M}})$,\;$w\in (L^2(\Lambda(\tilde{N})), d_{\tilde{N}})$, define $T_{\tilde{f}}'$ by
\[
\int_{\tilde{M}} v\wedge T_{\tilde{f}} w = \int_{\tilde{N}} T_{\tilde{f}} v \wedge w,
\]
then there exists a bounded operator $\mathscr{Y}$ with finite propagation, such that
\[
1-T_{\tilde{f}}'T_{\tilde{f}}= d_{\tilde{N}} \mathscr{Y}+\mathscr{Y}d_{\tilde{N}}.
\]
The proof of this fact is essentially the same to the one of \cite[Lemma 2.4]{W2013a}.
\end{enumerate}
For the properties listed above, more details can be found for instance in \cite{HS92,W2013a}.
For notation simplicity, we will denote $T_{\tilde{f}}$ simply by $T_{f}$.


\subsection{Homotopy invariance of higher signature}\label{subsection Homotopy invariance of higher signature}
In this subsection, we sketch Higson and Roe's proof of homotopy invariance of the higher signatures for closed manifolds \cite{HR1, HR2}.

Let $X$ be a oriented, closed topological manifold of dimension $n$.
Let $\tilde{X}$ be a regular $G$-covering space of $X$.
Moreover, we have the following commutative diagram
\begin{equation}\label{main.sanjiao}
\begin{tikzcd}[column sep=small]
M  \arrow[rr, "f"] \arrow[dr,"\phi"'] & & N  \arrow[dl,"\psi"]\\
& X  &
\end{tikzcd}	
\end{equation}
where $f:M \to N$ is a smooth orientation-preserving homotopy equivalence between two oriented, closed smooth manifolds $M$ and $N$, $\phi$ and $\psi$ are continuous maps, such that $\phi=\psi\circ f$.

Let $\tilde{N}=\psi^*(\tilde{X})$ and $\tilde{M}=f^*(\tilde{N})=\phi^*(\tilde{X})$ be the pull back coverings.
Set
\[
d :=\begin{pmatrix}
d_{\tilde{M}} & 0\\
0 & d_{\tilde{N}}
\end{pmatrix}\text{~and~}S :=\begin{pmatrix}
S_{\tilde{M}} & 0\\
0 & -S_{\tilde{N}}
\end{pmatrix},
\]
which are operators acting on $L^2(\Lambda(\tilde{M})) \bigoplus L^2(\Lambda(\tilde{N}))$.
Recall Definition \ref{poincaredualityoperator}, it is easy to see that $(L^2(\Lambda(\tilde{M})) \bigoplus L^2(\Lambda(\tilde{N})),d,S)$ is a \emph{Hilbert-Poincar\'e complex}.

In the following, if $S$ is a \emph{Poincar\'e duality operator} of a complex of Hilbert spaces $(\mathcal{H},d)$ (Appendix \ref{section Poin duality oper}, \cite{HR1}), we suppress the complex $(\mathcal{H},d)$ from the notation and only say $S$ is a  \emph{Poincar\'e duality operator} of $d$.

The following lemma is an analogy of \cite[Lemma 4.36]{WXY16}.
\begin{lemma}\label{lemma homo poincare duality topic}
Use the notations as above.
Let $g:N\to M$ be a homotopy inverse of $f$, $S_{\tilde{N}}$ is chain homotopy equivalent to $T_f^*S_{\tilde{M}}T_f$ and $T_g S_{\tilde{M}} T_g^*$.
\end{lemma}
\begin{proof}
It is sufficient to prove that $S_{\tilde{N}}$ is chain homotopy equivalent to $T_f^*S_{\tilde{M}}T_f$.
Let $v, w\in L^2(\Lambda(N))$, then
\begin{eqnarray*}
& & \int_{\tilde{N}}  v\wedge S_{\tilde{N}} w- \int_{\tilde{N}} v\wedge T_f^* S_{\tilde{M}} T_f w \\
&=& i^{p(p-1)+[\frac{n}{2}]}(\int_{\tilde{N}} v\wedge  w- \int_{\tilde{N}} T_f v\wedge T_f w)\\
&=& i^{p(p-1)+[\frac{n}{2}]}(\int_{\tilde{N}} v\wedge  w- \int_{\tilde{N}} v\wedge T'_f T_f w)\\
&=& i^{p(p-1)+[\frac{n}{2}]}\int_{\tilde{N}} v\wedge (d_{\tilde{N}}\mathscr{Y}+\mathscr{Y}d_{\tilde{N}} w)\\
&=& \int_{\tilde{N}} v\wedge S_{\tilde{N}}(d_{\tilde{N}}\mathscr{Y}+\mathscr{Y}d_{\tilde{N}} w	).
\end{eqnarray*}
Thus, we have $S_{\tilde{N}}-T_f^* S_{\tilde{M}} T_f= -d_{\tilde{N}}^* (S_{\tilde{N}}\mathscr{Y})+ S_{\tilde{N}}\mathscr{Y} d_{\tilde{N}}$.
\end{proof}

Set $D:=d+d^*$, we have $D \pm S$ are invertible \cite[section 5]{HR1}.
Moreover, by Lemma \ref{lemma homo poincare duality topic}, for $t \in [0,1]$, we have the path
\[
\begin{pmatrix}
S_{\tilde{M}} & 0\\
0 & -(1-t)S_{\tilde{N}}-tT_f^*S_{\tilde{M}} T_f
\end{pmatrix}
\]
are \emph{Poincar\'e duality operators} (Definition \ref{poincaredualityoperator}, \cite[Lemma 3.4]{HR1}, and \cite[Definition 3.1]{HR1}) of $d$.
Also, the path
\begin{equation}\label{eq.1}
\begin{pmatrix}
\cos( t\frac{\pi}{2})S_{\tilde{M}} & \sin(t\frac{\pi}{2}) S_{\tilde{M}} T_f\\
\sin(t\frac{\pi}{2}) T_f^*S_{\tilde{M}} & -\cos(t\frac{\pi}{2}) T_f^* S_{\tilde{M}} T_f
\end{pmatrix}	
\end{equation}
are \emph{Poincar\'e duality operators} of $d$.
Moreover, we can write the homotopy inverse of Eq.(\ref{eq.1}) as
\[
\begin{pmatrix}
\cos( t\frac{\pi}{2})S_{\tilde{M}} & \sin(t\frac{\pi}{2}) S_{\tilde{M}} T_g^*\\
\sin(t\frac{\pi}{2})T_g S_{\tilde{M}} & -\cos(t\frac{\pi}{2}) T_g S_{\tilde{M}} T_g^*
\end{pmatrix}.
\]

Note that, for $t\in [0,1]$, using the path (\emph{Poincar\'e duality operators} of $d$)
\[
\begin{pmatrix}
0& e^{it \pi}S_{\tilde{M} T_f}\\
e^{-it\pi} T_f^*S_{\tilde{M}} & 0
\end{pmatrix}
\]
we can be connected $
\begin{pmatrix}
0& S_{\tilde{M}}T_f\\
T_f^*S_{\tilde{M}} & 0
\end{pmatrix}$ to $\begin{pmatrix}
0&- S_{\tilde{M}} T_f\\
-T_f^*S_{\tilde{M}} & 0
\end{pmatrix}$.
After re-parameterize the $t$ variable, we will denote the above path by $S_t$.

For $t\in[0,1]$ and $n$ is even, we consider the operator
\begin{equation}\label{eq.3.3a}
P_+(D+ S_t)-P_+(D- S_t),	
\end{equation}
while $n$ is odd, we consider the invertible operator
\begin{equation}\label{eq.3.3b}
(D+S_t)(D-S_t)^{-1}
\end{equation}
which restricts to even forms.

The following two lemmas shows that $K$-theory of the $C^*$-algebra $C^*(\tilde{X})^G$ is the receptacle for the indices given by operators in Eq.(\ref{eq.3.3a}) and Eq.(\ref{eq.3.3b}).

\begin{lemma}[{\cite[Lemma 5.7]{HR1}}]\label{lemma odd functional cal for perturb}
Let $i=\sqrt{-1}$, for any $t$, the operators $i\pm (D\pm S_t)$ are invertible, and $(i\pm (D\pm S_t))^{-1}$ belongs to $C^*(\tilde{X})^G$.
\end{lemma}
\begin{prf}
Since $i\pm (D\pm S_t)$ is a self-adjoint operator and self-adjoint operator plus or minus $i$ are always invertible.

Hence, we have
\begin{eqnarray*}
(D+S_t+i)&=&(S_t(D+i)^{-1}+1)(D+i)\\
&=&	(D+i)((D+i)^{-1}S_t+1).
\end{eqnarray*}
Moreover, for any compact support $\phi$, we have
\begin{eqnarray*}
\phi (D+S_t+i)^{-1}&=& \phi ((S_t(D+i)^{-1}+1)(D+i))^{-1}\\
&=&\phi (D+i)^{-1}(S_t(D+i)^{-1}+1)^{-1},\\
(D+S_t+i)^{-1}\phi &=& ((S_t(D+i)^{-1}+1)(D+i))^{-1}\phi\\
&=& (S_t(D+i)^{-1}+1)^{-1}(D+i)^{-1}\phi.
\end{eqnarray*}
Operators $(S_t(D+i)^{-1}+1)$ are certainly bounded, and $\phi (D+i)^{-1},(D+i)^{-1}\phi $ are always compact.
Thus, $(i\pm (D\pm S_t))^{-1}$ is locally compact.

Without loss of generality, we suppose $T_f$ have finite propagation, thus $(S_t(D+i)^{-1}+1)$ can be approximated by finite propagation operators, and this is also true for the operator $(D+S_t+i)^{-1}= ((S_t(D+i)^{-1}+1)(D+i))^{-1}$.
\end{prf}

\begin{lemma}[{\cite[Lemma 5.8]{HR1}}]\label{lemma even functional cal for perturb}
For any bounded function $g:(-\infty, \infty)\to \mathbb{R}$, we have $g(D\pm S_t)-g(D)$ belongs to $C^*(\tilde{X})^G$.
\end{lemma}
\begin{prf}
It suffices to prove the lemma for $g(x) = \frac{x}{\sqrt{1+x^2}}$.
Using the integral
\[
\frac{x}{\sqrt{1+x^2}}=\frac{1}{\pi}\int_{1}^{\infty} \frac{s}{\sqrt{s^2-1}} x (x^2+s^2)^{-1}ds.
\]
We have
\begin{align*}
&g(D)-g(D+S_t)\\
=& \frac{1}{\pi}\int_1^\infty \frac{s}{\sqrt{s^2-1}}\left((D+is)^{-1}S_t (D+S_t+is)^{-1}\right)ds\\
 &+ \frac{1}{\pi}\int_1^\infty \frac{s}{\sqrt{s^2-1}}\left((D-is)^{-1}S_t(D+S_t-is)^{-1})\right)ds.
\end{align*}
\end{prf}

In summary, in the $K$-theory of the $C^*$-algebra $C^*(\tilde{X})^G$, we have an explicit path connecting
\[
[P_+(D+ S)]-[P_+(D- S)]	
\]
and
\[
[(D + S)(D-S)^{-1}]
\] 
to the trivial element respectively.


\bigskip
Now, we study the case that $X$, $M$, and $N$ are not compact.
Let us first recall the definition of \emph{Controlled homotopy equivalence}.

\begin{definition}\label{definition ana control homo}
Let $f: M\to N$ be a smooth homotopy equivalence. Fix a metric on $X$ that agrees with the topology of $X$.
Suppose $g: N\to M$ is a homotopy inverse of $f$.
Let $\{h_s \}_{0 \leq s\leq 1}$ be a homotopy between $f \circ g$ and $Id:N \to  N$.
Similarly, let $\{h'_s \}_{ 0 \leq s\leq 1}$ be a homotopy between $ g \circ f$ and $Id:M\to M$.
 We say $f$ is controlled over $X$ if there exist proper continuous maps $\phi: M \to X$ and $\psi: N\to X$ such that 
\begin{enumerate}
\item The following diagram commutes
\begin{equation*}
\begin{tikzcd}[column sep=small]
M  \arrow[rr, "f"] \arrow[dr,"\phi"'] & & N  \arrow[dl,"\psi"]\\
& X&
\end{tikzcd}
\end{equation*}
\item The diameter of the set $\psi h(a)=\{\psi( h_s(a))~\mid~  0\leq s\leq 1 \}$ is bounded by a constant, say $C$, uniformly for all $a\in N$;
\item The diameter of the set $\phi h'(b)=\{\phi( h'_s(b))~\mid~  0\leq s\leq 1  \}$ is bounded by a constant, say $C$, uniformly for all $b\in M$.
\end{enumerate}
\end{definition}

\begin{remark}
Under the circumstances of Eq.(\ref{main.sanjiao}).
If we further suppose $G=\pi_1X$ is the fundamental group of $X$ and $\tilde{X}$ is the universal cover of $X$.
Let $\tilde{N}=\psi^*(\tilde{X})$ and $\tilde{M}=f^*(\tilde{N})=\phi^*(\tilde{X})$ be the pull back coverings.
Then the lifting map $\tilde{f}:\tilde{M}\to \tilde{N}$ is automatically \emph{controlled homotopy equivalence} over $\tilde{X}$.
\end{remark}
With essentially the same argument as the compact case, we have.
\begin{proposition}\label{proposition Ana con hom kill sig}
With the notations as above.
If $f: M\to N$ is an \emph{controlled homotopy equivalence} over $X$.
Then
\[
\phi_*(\text{Ind}(D_M))=\psi_*(\text{Ind}(D_N)) \in K_*(C^*(X)),
\]
where, $D_M$ and $D_N$ denote the signature operator on $M$ and $N$ respectively.
\end{proposition}
Proposition \ref{proposition Ana con hom kill sig} has also been obtained by Kaminker-Miller \cite{KM85}, Kasparov \cite{K95}, Hilsum-Skandalis \cite{HS92}. 
However, we emphasize that the proof given in the current section follows from \cite{HR1} and \cite{HR2}.

\section{Index map for $L$-group}
\label{subsection Define for element of L}
In this section, using the idea of relative index, we explore the index map for $L$-group.
For each element in $L_n(\pi_1 X)$ (Appendix \ref{section WXY surgery sequence} or \cite[Definition 3.9]{WXY16} for details) we associate to it a $K$-theory class in $K_n(C^*(\tilde{X})^G)$.
We will show that this map is a well defined group homomorphism.

\subsection{Relative index}\label{subsection M infty and N infty}

The purpose of this subsection is to define the index of an element in $L_n(\pi_1 X)$.
We would like to introduce Lemma \ref{lemma almost invertible cite}, following \cite{HS92}, and Lemma \ref{lemma almost invertible} first. They will play central role in our construction.

\begin{lemma}[\cite{HS92}, Lemma 2.2]\label{lemma almost invertible cite}
Let $\epsilon$ and $K$ be two real numbers such that $4\epsilon K^2< 1$, $\nabla $ is a closed unbounded adjointable operator on Hilbert space $\mathcal{H}$ such that $\text{ran}\nabla  \subset \text{dom} \nabla$ and $\|\nabla^2\|\leq \epsilon^2$.
Moreover, there exists $x$, $y$ in $B(\mathcal{H})$ such that $x\nabla +\nabla  x=y$, where $y$ is invertible, $x \text{dom}\nabla  \subset \text{dom}\nabla $, $\|x\|\leq K$ and $\|y^{-1}\|\leq K$, then $\nabla +\nabla^*$ is invertible.
\end{lemma}
With Lemma \ref{lemma almost invertible cite} in hand, we can prove the following lemma.
\begin{lemma}\label{lemma almost invertible}
For $i=1,2$, let $d_i$ be close unbounded operator on Hilbert space $\mathcal{H}_i$.
If there exist bounded operators $F$, $G$, $y$, $z$, and numbers $\epsilon\geq 0$, $K\geq 1$ satisfies the following properties:
\begin{enumerate}
\item The norm of $F$, $G$, $y$ and $z$ are all less than $K$;
\item For $i=1,2$. $\text{ran~}d_i  \subset \text{dom~} d_i$ and $\|d_i^2\|\leq \epsilon^2$;
\item $F(\text{dom}~d_2)\subset \text{dom}~d_1$, $G(\text{dom}~d_1)\subset \text{dom}~d_2$, and $y (\text{dom}~d_1)\subset {dom}~d_1$, $z \text{dom}~d_2\subset {dom}~d_2$;

\item $\|d_1F-Fd_2\|\leq \epsilon^2$, and $\|d_2G-Gd_1\|\leq \epsilon^2$;
\item $\|1-FG-d_1 y -y d_1 \|\leq \epsilon^2$, and $\|1-GF-d_2 z -z d_2 \|\leq \epsilon^2$.
\end{enumerate}

Then 	
\[
\left(
\begin{matrix}
d_1+d_1^* & \alpha F\\
 \alpha F^*   & -d_2 -d_2^*
\end{matrix}
\right)
\]
is invertible, when $ \alpha \leq \frac{1}{2}K^{-2}(1+K)^{-2}$ and $\epsilon \leq \frac{1}{2}\sqrt{\frac{ \alpha}{2 \alpha+1}} $.
\end{lemma}
\begin{prf}
For 	
\[
\left(
\begin{matrix}
d_1 &  \alpha F\\
0  & -d_2	
\end{matrix}
\right),~\text{~consider~}~
\left(
\begin{matrix}
y & 0\\
 \alpha^{-1} G & -z	
\end{matrix}
\right).
\]
Then we have
\begin{align*}
& \begin{pmatrix}
y & 0\\
 \alpha^{-1} G & -z	
\end{pmatrix}
\begin{pmatrix}
d_1 &  \alpha  F\\
0  & -d_2
\end{pmatrix}+\begin{pmatrix}
d_1 &  \alpha F\\
0  & -d_2
\end{pmatrix}
\begin{pmatrix}
y & 0\\
 \alpha^{-1} G & -z	
\end{pmatrix}\\
=& \begin{pmatrix}
yd_1 + d_1 y + FG &  \alpha (yF-Fz) \\
 \alpha^{-1} (Gd_1-d_2 G) & GF+z d_2 + d_2z 	
\end{pmatrix}.	
\end{align*}

If $ \alpha \leq \frac{1}{2}K^{-2}(1+K)^{-2}$ and $\epsilon \leq \frac{1}{2}\sqrt{\frac{ \alpha}{2 \alpha+1}} $, we have the following estimate
\begin{align*}
& \left\|\begin{pmatrix}
1 & 0\\
0 & 1	
\end{pmatrix}- \begin{pmatrix}
yd_1 + d_1 y + FG &  \alpha (yF-Fz) \\
 \alpha^{-1} (Gd_1-d_2 G) & GF+z d_2 + d_2z 	
\end{pmatrix}\right\|\\
=& \left\|\begin{pmatrix}
1-yd_1 - d_1 y - FG & -  \alpha (yF-Fz) \\
-  \alpha^{-1} (Gd_1-d_2 G) & 1- GF- z d_2 - d_2z 	
\end{pmatrix}\right\| \leq \frac{1}{2}.	
\end{align*}
Thus	
\[
\begin{pmatrix}
y & 0\\
 \alpha^{-1} G & -z	
\end{pmatrix}\begin{pmatrix}
d_1 &  \alpha  F\\
0  & -d_2
\end{pmatrix}+\begin{pmatrix}
d_1 &  \alpha F\\
0  & -d_2
\end{pmatrix}\begin{pmatrix}
y & 0\\
 \alpha^{-1} G & -z	
\end{pmatrix}
\]
is invertible, and		
\[
\left\|\left( \left(
\begin{matrix}
y & 0\\
 \alpha^{-1} G & -z	
\end{matrix}
\right)\left(
\begin{matrix}
d_1 &  \alpha  F\\
0  & -d_2
\end{matrix}
\right)+
\left(
\begin{matrix}
d_1 &  \alpha F\\
0  & -d_2
\end{matrix}
\right)\left(
\begin{matrix}
y & 0\\
 \alpha^{-1} G & -z	
\end{matrix}
\right)\right)^{-1}\right\|<2.
\]

Furthermore, we have	
\[	
\left\|\left(
\begin{matrix}
y & 0\\
 \alpha^{-1} G & -z	
\end{matrix}
\right)\right\|\\
\leq 2K+2K^3(1+K)^2.
\]

Then by Lemma \ref{lemma almost invertible cite}, we see that
\[
\left(
\begin{matrix}
d_1+d_1^* &  \alpha F\\
 \alpha F^*   & -d_2-d_2^*
\end{matrix}
\right)
\]
is invertible.	
\end{prf}

\medskip

Before continuing, we introduce some notations which will be used in the rest of this section in the following remark.

\begin{remark}[Notation]\label{re.notation}
Suppose $\theta=(M,\partial M,\phi, N, \partial N,\psi, f)\in L_n(\pi_1 X)$.
Let $\tilde{X}$ be the universal cover of $X$, then $\tilde{X}$ is a regular $G$-covering of $X$ where $G=\pi_1 X$ is the fundamental group of $X$.
Let $M_\infty$ and $N_\infty$ be the manifolds with cylindrical ends associated to $M$ and $N$, i.e., $M_\infty = M \sqcup_{\partial M} \partial M \times [1,\infty)$ and $N_\infty=N \sqcup_{\partial N} \partial N \times [1,\infty)$.
Denote by $\tilde{M}_\infty$ (resp. $\tilde{N}_\infty $) the corresponding $G$-covering of $\Phi :M_\infty\to X\times [1,\infty)$ (resp. $\Psi :N_\infty\to X\times [1,\infty)$).
\end{remark}

Here we give briefly our motivation behind the whole discussion in the rest of this subsection.

\begin{remark}[Motivation]
In the setting of Hilbert-Poincar{\'e} complexes, as we explained in Section \ref{homotopy invariance of higher signature} and \cite{HR1}, a natural way to define the higher index of an element in $L_n(\pi_1 X)$ is by
\[
(D+ S)(D-S)^{-1}
\]
when $n$ is odd.
Unfortunately, this element generally does not represent a class in  $K_1(C^*(\tilde{X})^G)$.
Instead, we will construct ``almost Poincar\'e duality operators" $S_f$ and $S_f'$ (the precise definition of $S_f$ and $S_f'$ will be given below in this subsection) and define the index map for $L_n(\pi_1 X)$ in terms of $(D+\alpha S_f)(D-\alpha S'_f)^{-1}$ to fix this problem. Intuitively, to construct operator $S_f$ and $S'_f$ it essentially amounts to connecting the Poincar\'e duality $S_{\partial M}$ on $\partial M$ to $-S_{\partial M}$. 

Honestly, such a definition is quite unnatural, since $S_f$ or $S'_f$ are not Poincar{\'e} duality operators themselves (in fact not chain complex morphisms in general). However, one can show that the above element $(D+\alpha S_f)(D-\alpha S'_f)^{-1}$ represents a class in  $K_1(C^*(\tilde{X})^G)$.	
\end{remark}

\medskip



Consider manifolds $\partial M \times \mathbb{R}$ and $\partial N \times \mathbb{R}$. We first focusing on the case where the dimension of $M$ and $N$ are odd. The construction for even case are similar. 
\paragraph{Odd case}
We will construct bounded operators $S'_{\partial f\times \mathbb{R}}$ and $S_{\partial f\times \mathbb{R}}$ in $B(L^2(\Lambda(\partial M \times \mathbb{R})) \bigoplus L^2(\Lambda(\partial N \times \mathbb{R})))$, which can be viewed as ``almost Poincar\'e duality operator".

Note that $\partial f: \partial M \to \partial N $ is a homotopy equivalent.
And any form $\omega\in L^2(\Lambda(\partial M \times \mathbb{R})) \bigoplus L^2(\Lambda(\partial N \times \mathbb{R})) $ can be decomposed as
 \[
 \omega_1 + \omega_2 \wedge dt,
 \]
where both $\omega_1, \omega_2$ are $L^2$ forms along $\partial M \sqcup \partial N$.
Let $(x,t)$ be a point in $\partial M\times \mathbb{R}\sqcup \partial N\times \mathbb{R}$.
Let $N_i$ ($ i=0,1,2,3$) be positive integers, such that $\frac{1}{N_i-N_{i-1}}$ are sufficiently small  for $i=1,2,3$. Recall that for now, the dimension of $\partial M $ and $\partial N$ are even. 
We define $S'_{\partial f\times \mathbb{R}}$ by setting $S'_{\partial f\times \mathbb{R}}\omega|_{(x,t)}$ as follows:
\begin{enumerate}
\item when $ -\infty <t \leq N_0$, set $S'_{\partial f\times \mathbb{R}}\omega|_{(x,t)}$ as
    \[
    [(\begin{pmatrix}
          S_{\partial M} &0\\			
          0& -S_{\partial N}	
          \end{pmatrix}\omega_1 )\wedge dt+
      \begin{pmatrix}    
          S_{\partial M} &0\\			    
          0& -S_{\partial N}  
      \end{pmatrix}\omega_2]|_{(x,t)}; 
  \]
\item when $N_0 < t \leq N_1$, set $S'_{\partial f\times \mathbb{R}}\omega|_{(x,t)}$ as
\begin{eqnarray*}   
&&	[(\begin{pmatrix} 
S_{\partial M} &0\\   
0& -  \frac{N_1-t}{N_1-N_0}  S_{\partial N}- \frac{t-N_0}{N_1-N_0} T_{\partial f}^* S_{\partial\tilde{ M}} T_{\partial f}
\end{pmatrix}\omega_1)\wedge dt \\ 
&& + \begin{pmatrix} 
S_{\partial M} &0\\ 
0& -  \frac{N_1-t}{N_1-N_0}  S_{\partial N}- \frac{t-N_0}{N_1-N_0}  T_{\partial f}^* S_{\partial M} T_{\partial f}\end{pmatrix}\omega_2 ]|_{(x,t)};    
\end{eqnarray*}   
\item when $N_1< t \leq N_2$, set $S'_{\partial f\times \mathbb{R}}\omega|_{(x,t)}$ as  
  \begin{eqnarray*}
&&	[(\begin{pmatrix}
\cos((\frac{t-N_1}{N_2-N_1}) \frac{\pi}{2})S_{\partial M} &  \sin((\frac{t-N_1}{N_2-N_1}) \frac{\pi}{2}) S_{\partial M} T_{\partial f}\\
\sin((\frac{t-N_1}{N_2-N_1})\frac{\pi}{2}) T_{\partial f}^*S_{\partial M} & - \cos((\frac{t-N_1}{N_2-N_1}) \frac{\pi}{2}) T_{\partial f}^* S_{\partial M} T_{\partial f}
\end{pmatrix} \omega_1)\wedge dt \\ &&+  \begin{pmatrix}
\cos((\frac{t-N_1}{N_2-N_1}) \frac{\pi}{2})S_{\partial M} &  \sin((\frac{t-N_1}{N_2-N_1}) \frac{\pi}{2}) S_{\partial M} T_{\partial f}\\
\sin((\frac{t-N_1}{N_2-N_1})\frac{\pi}{2}) T_{\partial f}^*S_{\partial M} & - \cos((\frac{t-N_1}{N_2-N_1}) \frac{\pi}{2}) T_{\partial f}^* S_{\partial M} T_{\partial f}\end{pmatrix} \omega_2 ]|_{(x,t)} ;
\end{eqnarray*}
\item when $N_2< t < +\infty$, set $S'_{\partial f\times \mathbb{R}}\omega|_{(x,t)}$ as
\[
[ (\begin{pmatrix}
0 & S_{\partial M} T_{\partial f}\\
T_{\partial f}^*S_{\partial M}&0
\end{pmatrix}\omega_1 )\wedge dt+ \begin{pmatrix}
0 & S_{\partial M}T_{\partial f}\\
T_{\partial f}^*S_{\partial M}&0
\end{pmatrix} \omega_2 ]|_{(x,t)}.
\]
\end{enumerate}
In the meanwhile,  we define $S_{\partial f\times \mathbb{R}}$ by setting $S_{\partial f\times \mathbb{R}}\omega|_{(x,t)}$ as follows:
\begin{enumerate}
	\item when $-\infty< t \leq N_2$, set $S_{\partial f\times \mathbb{R}}\omega|_{(x,t)}$ as $S'_{\partial f\times \mathbb{R}}\omega|_{(x,t)}$;
	\item when $N_2< t \leq N_3$, set $S_{\partial f\times \mathbb{R}}\omega|_{(x,t)}$ as
    \begin{eqnarray*}
    	 &&[ (\begin{pmatrix}
0 & e^{i\frac{t-N_2}{N_3-N_2} \pi} S_{\partial M}T_{\partial f}\\
e^{-i\frac{t-N_2}{N_3-N_2} \pi}T_{\partial f}^*S_{\partial M}&0
\end{pmatrix}\omega_1 )\wedge dt  \\ &&+ \begin{pmatrix}
0 & e^{i\frac{t-N_2}{N_3-N_2} \pi} S_{\partial M} T_{\partial f}\\
e^{-i\frac{t-N_2}{N_3-N_2} \pi} T_{\partial f}^*S_{\partial M}&0
\end{pmatrix} \omega_2 ]|_{(x,t)};
    \end{eqnarray*}
  \item  when $N_3< t < +\infty$, set $S_{\partial f\times \mathbb{R}}\omega|_{(x,t)}$ as
\[
    	 [ (  \begin{pmatrix}
0 & -S_{\partial M} T_{\partial f}\\
-T_{\partial f}^*S_{\partial M}&0
\end{pmatrix}\omega_1 )\wedge dt   +   \begin{pmatrix}
0 & -S_{\partial M} T_{\partial f}\\
-T_{\partial f}^*S_{\partial M}&0
\end{pmatrix} \omega_2 ]|_{(x,t)}.
\]
\end{enumerate}

\begin{remark}\label{replace smooth functions}
\begin{enumerate}
  \item In the above construction, positive integers $N_i$ ($ i=0,1,2,3$) can be chosen arbitrarily.
  \item We can replace all the functions parameterize by $t$ in the above construction by smooth functions.
  \item Without loss of generality, we can assume $S'_{\partial f\times \mathbb{R}}$ and $S_{\partial f\times \mathbb{R}}$ are smoothly defined.
\end{enumerate}

\end{remark}


We define operators $S'_f$ and $S_f$ on $L^2(\Lambda (M_\infty)) \bigoplus L^2(\Lambda (N_\infty))$ as
\begin{eqnarray*}
	S'_f& = & \begin{pmatrix}
S_{M_\infty} & 0\\
0& -S_{N_\infty}
\end{pmatrix}\begin{pmatrix}
1-\chi_{\partial M\times [N_0, +\infty)} &0\\			
0& 1-\chi_{\partial N\times [N_0, +\infty)}	
\end{pmatrix}\\&& + S'_{\partial f\times \mathbb{R}}\begin{pmatrix}
\chi_{\partial M\times [N_0, +\infty)} &0\\			
0& \chi_{\partial N\times [N_0, +\infty)}	
\end{pmatrix};\\
S_f& = & \begin{pmatrix}
S_{M_\infty} & 0\\
0& -S_{N_\infty}
\end{pmatrix}\begin{pmatrix}
1-\chi_{\partial M\times [N_0, +\infty)} &0\\			
0& 1-\chi_{\partial N\times [N_0, +\infty)}	
\end{pmatrix}\\&& + S_{\partial f\times \mathbb{R}}\begin{pmatrix}
\chi_{\partial M\times [N_0, +\infty)} &0\\			
0& \chi_{\partial N\times [N_0, +\infty)}	
\end{pmatrix}.
\end{eqnarray*}

It is not hard to prove that $S_f$ and $S'_f$ are bounded operators on Hilbert space $L^2(\Lambda (M_\infty)) \bigoplus L^2(\Lambda (N_\infty))$.

\medskip


With the above construction and the notations from Remark \ref{re.notation}, we have $G$-equivariant versions of $S_f$ and $S'_f$, which are bounded operators on Hilbert space $L^2(\Lambda (\tilde{M}_\infty)) \bigoplus L^2(\Lambda (\tilde{N}_\infty))$.
We will denote by $S_f$ and $S'_f$ as well.

The following Proposition is a corollary of Lemma \ref{lemma almost invertible}.
\begin{proposition}\label{proposition almost invertible}
With the notations as above, for some proper chosen $\alpha>0$, we have
\[
\begin{pmatrix}
\left(\begin{matrix}
d_{\tilde{M}_\infty}+d^*_{\tilde{M}_\infty}& 0\\
0 & d_{\tilde{N}_\infty} + d^*_{\tilde{N}_\infty}	
\end{matrix}\right) & \alpha S'_f\\			
\alpha S'_f& \left(\begin{matrix}
d_{\tilde{M}_\infty}+d^*_{\tilde{M}_\infty}& 0\\
0 & d_{\tilde{N}_\infty} + d^*_{\tilde{N}_\infty}
\end{matrix}\right)		
\end{pmatrix},
\]
and
\[
\begin{pmatrix}
\left(\begin{matrix}
d_{\tilde{M}_\infty}+d^*_{\tilde{M}_\infty}& 0\\
0 & d_{\tilde{N}_\infty} + d^*_{\tilde{N}_\infty}	
\end{matrix}\right) & \alpha S_f\\			
\alpha S_f& \left(\begin{matrix}
d_{\tilde{M}_\infty}+d^*_{\tilde{M}_\infty}& 0\\
0 & d_{\tilde{N}_\infty} + d^*_{\tilde{N}_\infty}
\end{matrix}\right)		
\end{pmatrix},
\]
are invertible.
\end{proposition}
\begin{prf}
We will prove the $S_f$ case only. 
In our argument, $S_f$ takes the role of $F$ in Lemma \ref{lemma almost invertible}.
Thanks to Lemma \ref{lemma almost invertible}.
It is sufficient to show the existence of $G$, $y$, $z$ and $K \geq 1$ which satisfying the conditions in Lemma \ref{lemma almost invertible}.

We first prove the existence of $G$.
Recall that the restricting of $f$ on the boundary $\partial f: \partial M \to \partial N$ is homotopy equivalence.
 Suppose $\partial g: \partial N \to \partial M$ to be the homotopy inverse of  $\partial f$.
 Similarly, we can define an operator $S'_{\partial g\times \mathbb{R}}$ in  $B(L^2(\Lambda(\partial \tilde{M} \times \mathbb{R})) \bigoplus L^2(\Lambda(\partial \tilde{N} \times \mathbb{R})) )$ as follows:
\begin{enumerate}
	\item when $ -\infty <t \leq N_0$, set $S'_{\partial g\times \mathbb{R}}\omega|_{(x,t)}$ as
	\[
	[(\begin{pmatrix}
S_{\partial\tilde{ M}} &0\\			
0& -S_{\partial\tilde{ N}}	
\end{pmatrix}\omega_1 )\wedge dt+
\begin{pmatrix}
S_{\partial\tilde{ M}} &0\\			
0& -S_{\partial\tilde{ N}}	
\end{pmatrix}\omega_2]|_{(x,t)};
	\]
	\item when $N_0 < t \leq N_1$, set $S'_{\partial g\times \mathbb{R}}\omega|_{(x,t)}$ as
	\begin{eqnarray*}
	&&	[(\begin{pmatrix}
S_{\partial\tilde{ M}} &0\\
0& -  \frac{N_1-t}{N_1-N_0}  S_{\partial\tilde{ N}}- \frac{t-N_0}{N_1-N_0} T_{\partial g}S_{\partial\tilde{ M}} T_{\partial g}^*
\end{pmatrix}\omega_1)\wedge dt \\
&& + \begin{pmatrix}
S_{\partial\tilde{ M}} &0\\
0& -  \frac{N_1-t}{N_1-N_0}  S_{\partial\tilde{ N}}- \frac{t-N_0}{N_1-N_0}  T_{\partial g}S_{\partial\tilde{ M}}T_{\partial g}^*
\end{pmatrix}\omega_2 ]|_{(x,t)};
	\end{eqnarray*}
	\item when $N_1< t \leq N_2$, set $S'_{\partial g\times \mathbb{R}}\omega|_{(x,t)}$ as
	\begin{eqnarray*}
	&&	[(\begin{pmatrix}
\cos((\frac{t-N_1}{N_2-N_1}) \frac{\pi}{2})S_{\partial\tilde{M}} &  \sin((\frac{t-N_1}{N_2-N_1}) \frac{\pi}{2}) S_{\partial \tilde{M}} T_{\partial g}^*\\
\sin((\frac{t-N_1}{N_2-N_1})\frac{\pi}{2}) T_{\partial g}S_{\partial \tilde{M}} & - \cos((\frac{t-N_1}{N_2-N_1}) \frac{\pi}{2}) T_{\partial g} S_{\partial \tilde{M}}T_{\partial g}^*
\end{pmatrix} \omega_1)\wedge dt \\ &&+  \begin{pmatrix}
\cos((\frac{t-N_1}{N_2-N_1}) \frac{\pi}{2})S_{\partial\tilde{M}} &  \sin((\frac{t-N_1}{N_2-N_1}) \frac{\pi}{2}) S_{\partial \tilde{M}}T_{\partial g}^*\\
\sin((\frac{t-N_1}{N_2-N_1})\frac{\pi}{2})T_{\partial g}^*S_{\partial \tilde{M}} & - \cos((\frac{t-N_1}{N_2-N_1}) \frac{\pi}{2}) T_{\partial g}S_{\partial \tilde{M}} T_{\partial g}^*
\end{pmatrix} \omega_2 ]|_{(x,t)} ;
	\end{eqnarray*}
	\item when $N_2< t < +\infty$, set $S'_{\partial g\times \mathbb{R}}\omega|_{(x,t)}$ as
    \[
   [ (\begin{pmatrix}
0 & S_{\partial \tilde{M}} T_{\partial g}^*\\
T_{\partial g}S_{\partial \tilde{M}}&0
\end{pmatrix}\omega_1 )\wedge dt+ \begin{pmatrix}
0 & S_{\partial \tilde{M}} T_{\partial g}^*\\
T_{\partial g}S_{\partial \tilde{M}}&0
\end{pmatrix} \omega_2 ]|_{(x,t)}.
    \]
\end{enumerate}
 
Then we define an operator $S_g$ in $B(L^2(\Lambda(\tilde{M}_\infty )) \oplus L^2(\Lambda( \tilde{N}_\infty)))$ as:
\begin{eqnarray*}
S_g & = & \begin{pmatrix}
S_{\tilde{M}_\infty} & 0\\
0& -S_{\tilde{N}_\infty}
\end{pmatrix}\begin{pmatrix}
1-\chi_{\partial\tilde{ M}\times [N_0, +\infty)} &0\\			
0& 1-\chi_{\partial\tilde{ N}\times [N_0, +\infty)}	
\end{pmatrix}\\&& + S_{\partial g\times \mathbb{R}}\begin{pmatrix}
\chi_{\partial\tilde{ M}\times [N_0, +\infty)} &0\\			
0& \chi_{\partial\tilde{ N}\times [N_0, +\infty)}	
\end{pmatrix}.
\end{eqnarray*}

By choosing the smooth functions in Remark \ref{replace smooth functions} so that their derivative functions all have small supremum-norm, say $\varepsilon$, we have
\[\|d_\infty S_f+S_f d^*_\infty\|\leq \varepsilon;\]
and
\[\|-d^*_\infty S_g-S_g d_\infty\|\leq \varepsilon,\]
where, $d_\infty=\begin{pmatrix}
	d_{\tilde{M}_\infty} & 0 \\
	0 & d_{\tilde{N}_\infty}
\end{pmatrix}$.

We now prove the existence of $y$ and $z$.
Again we define $y$ by setting $y (w)|_{(x,t)}$.
We will show the definition of $y (w)|_{(x,t)}, N_0 \leq t\leq N_1$ in details only.
The definition of $y (w)|_{x,t}$ for other $t$ are similar.

We consider $\partial \tilde{M}\times \mathbb{R}\sqcup \partial \tilde{N}\times \mathbb{R}$ first.
As shown in Lemma \ref{lemma homo poincare duality topic}, $S_{\partial \tilde{N}\times \mathbb{R}}$ is chain homotopy equivalent to $T_{\partial f\times \mathbb{R}}^* S_{\partial \tilde{M}\times \mathbb{R}} T_{\partial f\times \mathbb{R}}$ and $T_{\partial g\times \mathbb{R}} S_{\partial \tilde{M}\times \mathbb{R}} T_{\partial f\times \mathbb{R}}^*$.
It is easy to see that bounded operators $T_{\partial f\times \mathbb{R}}^* S_{\partial \tilde{M}\times \mathbb{R}} T_{\partial f\times \mathbb{R}}S_{\partial \tilde{N}\times \mathbb{R}}$, $S_{\partial \tilde{N}\times \mathbb{R}}T_{\partial f\times \mathbb{R}}^* S_{\partial \tilde{M}\times \mathbb{R}} T_{\partial f\times \mathbb{R}}$ and $T_{\partial f\times \mathbb{R}}^* S_{\partial \tilde{M}\times \mathbb{R}} T_{\partial f\times \mathbb{R}}T_{\partial g\times \mathbb{R}} S_{\partial \tilde{M}\times \mathbb{R}} T_{\partial f\times \mathbb{R}}^*$ are all chain homotopy equivalent to 1.
 Let $y_1$, $y_2$ and $y_3$ be the connecting map respectively. 

Then, for $N_0 \leq t\leq N_1$, we can define $y (w)|_{x,t}$ by

\[
\left(\frac{(t-N_0)(N_1-t)}{(N_1-N_0)^2}y_1 + \frac{(t-N_0)(N_1-t)}{(N_1-N_0)^2}y_2 +(\frac{t-N_0}{N_1-N_0})^2 y_3 \right)(w)|_{(x,t)}.
\]

%

With a little modification, we can see that in $B(L^2(\Lambda(\tilde{M}_\infty )) \bigoplus L^2(\Lambda( \tilde{N}_\infty)) )$ there exist bounded operators $y$ and $z$ satisfying
\[\|1-S_f S_g-d_\infty y- yd_\infty \|\leq C \varepsilon,\]
and
\[\|1-S_gS_f-d^*_\infty z-z d^*_\infty \|\leq C \varepsilon,\]
where $C$ is a constant.
Now it is easy to see, choosing the smooth functions in Remark \ref{replace smooth functions} so that their derivative functions all have small supremum-norm, $G$, $y$, $z$ will satisfying the conditions in Lemma \ref{lemma almost invertible}.
On the other hand, we can choose constant $K=\max\{\|S_f\|, \|S_g\|, \|y\|, \|z\|\}$.
Rather, change supremum-norm of the derivative functions do not increase $K$.
In summary, we have
\[
\begin{pmatrix}
\left(\begin{matrix}
d_{\tilde{M}_\infty}+d^*_{\tilde{M}_\infty}& 0\\
0 & d_{\tilde{N}_\infty} + d^*_{\tilde{N}_\infty}	
\end{matrix}\right) & \alpha S_f\\			
\alpha S_f& \left(\begin{matrix}
d_{\tilde{M}_\infty}+d^*_{\tilde{M}_\infty}& 0\\
0 & d_{\tilde{N}_\infty} + d^*_{\tilde{N}_\infty}
\end{matrix}\right)		
\end{pmatrix}
\]
is invertible for some properly chosen $\alpha$.
\end{prf}

\begin{corollary}\label{cor almost invertible}
For some proper chosen $\alpha>0$, bounded operators
\[
\left(
\begin{matrix}
d_{\tilde{M}_\infty}+d^*_{\tilde{M}_\infty}& 0\\
0 & d_{\tilde{N}_\infty} + d^*_{\tilde{N}_\infty}	
\end{matrix}
\right) \pm \alpha S'_f,
\]
and
\[
\left(
\begin{matrix}
d_{\tilde{M}_\infty}+d^*_{\tilde{M}_\infty}& 0\\
0 & d_{\tilde{N}_\infty} + d^*_{\tilde{N}_\infty}
\end{matrix}
\right) \pm \alpha S_f
\]
are invertible.
\end{corollary}
\begin{prf}
	Restricting
	\[
\begin{pmatrix}
\left(\begin{matrix}
d_{\tilde{M}_\infty}+d^*_{\tilde{M}_\infty}& 0\\
0 & d_{\tilde{N}_\infty} + d^*_{\tilde{N}_\infty}	
\end{matrix}\right) & \alpha S'_f\\			
\alpha S'_f& \left(\begin{matrix}
d_{\tilde{M}_\infty}+d^*_{\tilde{M}_\infty}& 0\\
0 & d_{\tilde{N}_\infty} + d^*_{\tilde{N}_\infty}
\end{matrix}\right)		
\end{pmatrix},
\]
 to the subspace
 \[
\bigvee \left\{\begin{pmatrix}v\\v\end{pmatrix}|v\in L^2(\Lambda (\tilde{M}_\infty)) \bigoplus L^2(\Lambda (\tilde{N}_\infty)) \right\}
 \]
we have $\begin{pmatrix}
d_{\tilde{M}_\infty}+d_{\tilde{M}_\infty}^*& 0\\
0 & d_{\tilde{N}_\infty} + d_{\tilde{N}_\infty}^*	
\end{pmatrix}+ \alpha S'_f$ is invertible.

And, restricting
\[
\begin{pmatrix}
\left(\begin{matrix}
d_{\tilde{M}_\infty}+d^*_{\tilde{M}_\infty}& 0\\
0 & d_{\tilde{N}_\infty} + d^*_{\tilde{N}_\infty}	
\end{matrix}\right) & \alpha S'_f\\			
\alpha S'_f& \left(\begin{matrix}
d_{\tilde{M}_\infty}+d^*_{\tilde{M}_\infty}& 0\\
0 & d_{\tilde{N}_\infty} + d^*_{\tilde{N}_\infty}
\end{matrix}\right)		
\end{pmatrix},
\]
to the subspace
\[
\bigvee \left\{\begin{pmatrix}v\\-v\end{pmatrix}|v\in L^2(\Lambda (\tilde{M}_\infty)) \bigoplus L^2(\Lambda (\tilde{N}_\infty)) \right\},
\]
we have $\begin{pmatrix}
d_{\tilde{M}_\infty}+d_{\tilde{M}_\infty}^*& 0\\
0 & d_{\tilde{N}_\infty} + d_{\tilde{N}_\infty}^*	
\end{pmatrix}+ \alpha S'_f$ is invertible.

Using the same arguments we can get
\[
\begin{pmatrix}
d_{\tilde{M}_\infty}+d_{\tilde{M}_\infty}^*& 0\\
0 & d_{\tilde{N}_\infty} + d_{\tilde{N}_\infty}^*	
\end{pmatrix}\pm \alpha S_f
\]
are invertible.
\end{prf}

Set 
\[
D=\begin{pmatrix}
d_{\tilde{M}_\infty}+d_{\tilde{M}_\infty}^*& 0\\
0 & d_{\tilde{N}_\infty} + d_{\tilde{N}_\infty}^*	
\end{pmatrix}.
\]
Following subsection \ref{subsection Homotopy invariance of higher signature},  we consider the following invertible operator
\[
(D+\alpha S_f)(D-\alpha S'_f)^{-1}
\]
from $L^2(\Lambda^{even}(\tilde{M}_\infty)\bigoplus \Lambda^{even}(\tilde{N}_\infty))$ to $L^2(\Lambda^{even}(\tilde{M}_\infty)\bigoplus  \Lambda^{even}(\tilde{N}_\infty))$, with $\alpha$ while chosen. 

Now, our main task is to show that the above invertible operator defines a class in $K_1(C^*(\tilde{X})^G$. Let $N_5\geq N_4$ are two integers bigger than $N_3$.
Let $\rho: M_\infty\sqcup N_\infty \to [0,1]$ a smooth function which is constant $1$ for $x\leq N_4$ and $0$ when $x$ is greater than $N_5$, always constant along $\partial M  \sqcup \partial N$ and $\rho'$ has small supremum-norm.
Then we have
\begin{eqnarray*}
	(D+\alpha S_f)(D-\alpha S'_f)^{-1}&=&1+\alpha(S'_f+S_f)\rho(D-\alpha S'_f)^{-1}\\
	&=& 1+\alpha(S'_f+S_f)(D-\alpha S'_f)^{-1}\rho \\
	&& +\alpha(S'_f+S_f)(D-\alpha S'_f)^{-1}[D,\rho](D-\alpha S'_f)^{-1}.
\end{eqnarray*}
 Since we can make the norm of 
\[
\alpha(S'_f+S_f)(D-\alpha S'_f)^{-1}[D,\rho](D-\alpha S'_f)^{-1}
\]
to be sufficiently small, one can see that $(D+\alpha S_f)(D-\alpha S'_f)^{-1}$ and $1+\alpha(S'_f+S_f)(D-\alpha S'_f)^{-1}\rho$ represent the same $K$-theory element. Furthermore, 
note that $\alpha(S_f+S_f)(D-\alpha S'_f)^{-1}$ belongs to $C^*(\tilde{X} \times [1,\infty))^G$, we can approximate it by an  invertible element $a$ in $C^*(\tilde{X} \times [1,\infty))^G$ with finite propagation, we may suppose the propagation of $a$ is bounded by $p$. Then obviously $1+\alpha a\rho$ defines an element in $K_1(C^*(\tilde{X}\times [1,N_5+p])^G)$. Thus one can see that 
\[(D+\alpha S_f)(D-\alpha S'_f)^{-1}\]
represent an element in $K_1(C^*(\tilde{X})^G)$, 
which will be denoted by  $\text{Ind}(\theta)$. We also omit $\alpha$ and no confusion shall be arose.

\paragraph{Even case}
When the dimension of $M$ and $N$ are even, the definition of   $S_f$ and $S_f'$ are totally similar. It will cause no confusion if we leave out the details. We thus consider the following well defined operator
\[
P_+(D+\alpha S_f)~\text{~and~}~P_+(D-\alpha S'_f).
\]
while $\alpha$ is well chosen.
Similarly, we can show that $[P_+(D+\alpha S_f)]-[P_+(D-\alpha S'_f)]$ actually defines an element in $K_0(C^*(\tilde{X})^G)$.

%
In fact, let $\rho: M_\infty\sqcup N_\infty \to [0,1]$ a smooth function which is constant $1$ for $x\leq N_3$ and $0$ when $x$ is greater than $N_3 +1$, constant along $\partial M  \sqcup \partial N$ and $\rho'$ has small supremum-norm.
Then the same as the proof of Lemma \ref{lemma even functional cal for perturb}, we have
\begin{align*}
&(D+\alpha S_f)\left(1+(D+\alpha S_f)^2\right)^{-1/2}-(D-\alpha S'_f)\left(1+(D-\alpha S'_f)^2\right)^{-1/2}\\
={}& \frac{1}{\pi}\int_1^\infty\frac{s}{\sqrt{s^2-1}}\left(R_1\rho(D-\alpha S_f+is)^{-1}+R_2\rho(D-\alpha S_f-is)^{-1}\right)ds\\
={}& \frac{1}{\pi}\int_1^\infty\frac{s}{\sqrt{s^2-1}}\left(R_1(D-\alpha S_f+is)^{-1}\rho + R_2(D-\alpha S_f-is)^{-1}\rho\right)ds\;+ \epsilon	
\end{align*}
where, $R_1 = (D+\alpha S_f+is)^{-1}(S'_f+S_f )$, $R_2 = (D+\alpha S_f-is)^{-1}(S'_f+S_f )$ and $\epsilon$ can be represent as following integral
\begin{align*}
&-\frac{1}{\pi}\int_1^\infty\frac{s}{\sqrt{s^2-1}}\left(R_1(D-\alpha S_f+is)^{-1}[D,\rho](D-\alpha S_f+is)^{-1}\right)ds\\
&-\frac{1}{\pi}\int_1^\infty\frac{s}{\sqrt{s^2-1}}\left(R_2(D-\alpha S_f-is)^{-1}[D,\rho](D-\alpha S_f+is)^{-1}\right)ds	
\end{align*}
from the above formula we see that $\epsilon$ ``only" dependents on $[D,\rho]$, so we can choose $\rho$ such that the norm of $\epsilon$ is sufficiently small.

Above computation tell us that
\[
[P_+(D+\alpha S_f)]-[P_+(D-\alpha S'_f)]
\]
defines an class in $K_0(C^*(\tilde{X}\times [1,N_3+1])^G)$, which is isomorphic to $ K_0(C^*(\tilde{X})^G)$.
We still denote this $K$-theory class in $K_0(C^*(\tilde{X})^G)$ by $[P_+(D+\alpha S_f)]-[P_+(D-\alpha S'_f)]$, and we think this abuse of notation would not cause any confusion.

In summary, we have a map $\text{Ind}: L_n (\pi_1 X)\to K_n(C^*(\tilde{X})^G)$.
We conclude this subsection by the following simple consequence.

\begin{theorem}\label{theorem homotopy trivial L}
Let $\theta=(M,\partial M,\phi, N, \partial N,\psi, f) \in L_n(\pi_1 X)$ and suppose $f :M \to N$ is a homotopy equivalent (not only restricting to the boundary).
Then $\text{Ind}(\theta)$ is trivial.
\end{theorem}

\subsection{Additivity}
In this subsection, we will continue our study of the index map $\text{Ind}: L_n (\pi_1 X)\to K_n(C^*(\tilde{X})^G)$.

Let $\theta$ and $\theta'$ be two elements in $L_n(\pi_1 X)$.
Recall, the addition of the abelian group $L_n(\pi_1 X)$ is given by disjoint union(\cite[Definition 3.10]{WXY16}, Definition \ref{definition eL}).
In the following theorem, we use a misleading notation, $\theta\sqcup \theta'$, denote the disjoint union of $\theta$ and $-\theta'$, which belongs to $L_n(\pi_1 X)$.
\begin{theorem}\label{theorem L differ}
Let $\theta, \theta'\in L_n(\pi_1 X)$.
We further suppose $\theta =(M, \partial M,\phi, N,\partial N, \psi, f )$, $\theta' = (M', \partial M', \phi', N', \partial N',\psi', f')$ and satisfies the following properties:
\begin{enumerate}
\item $\partial M=\partial M'$,\;$\partial N=\partial N'$, restricting to the boundary $f=f'$,\;$\phi=\phi'$ and $\psi=\psi'$;
\item There exist two manifolds with boundary $(W, \partial W)$ and $(V, \partial V)$, continuous maps $\Phi:W \to X$(resp. $\Psi:V \to X$).
Manifolds $W$ and $V$ both of dimension $n+1$.
Moreover, $\partial W = M\sqcup_{\partial M} M'$(resp. $\partial V = N\sqcup_{\partial N} N'$) and $\Phi$(resp. $\Psi$) restricts to $\phi\sqcup \phi'$(resp. $\psi\sqcup \psi'$) on $M\sqcup_{\partial M} M'$(resp. $N\sqcup_{\partial N} N'$);
\item There exist a degree one normal map $F:W\to V$ such that $\Psi \circ F=\Phi$.
Moreover, $F$ restricts to $f\sqcup f'$ on $M\sqcup_{\partial M} M'$.
\end{enumerate}

Then we have
\begin{equation}\label{form difference formila}
	\text{Ind}(\theta)-\text{Ind}(\theta')=\text{Ind}(\theta\sqcup \theta')-\text{Ind}(\theta' \sqcup \theta').
\end{equation}
\end{theorem}
\begin{prf}
We go through details for odd case only, the even case are totally the same.

For real numbers $T$ and $T_1 \leq T_2$, we denote $M_T:=M\sqcup_{\partial M} \partial M \times [0,T]$, $M_{T\geq}:=\partial M\times [T, \infty)$, and $M_{[T_1, T_2]}:=\partial M \times [T_1,T_2]$.
Similar, for $M'$, $N$ and $N'$, we introduce the same notation.

Take a real number $R$, we construct $M_R \sqcup_{\partial M} M'_R$ by pasting along $\partial M \times \{R\}\subset M\sqcup_{\partial M}\partial M\times [0,R]$ and $\partial M' \times \{R\}\subset M' \sqcup_{\partial M'}\partial M' \times [0,R]$.
Similar, for $N$ and $N'$, we introduce the same notation.
Without loss of generality, we can identify $M\sqcup_{\partial M} M'$ with $M_R\sqcup_{\partial M\times R} M'_R$, and identify $N\sqcup_{\partial N} N'$ with $N_R\sqcup_{\partial N\times R} N'_R$.

For real number $R\geq K$.
We will introduce several Hilbert spaces
\begin{enumerate}
\item $H_1$ stands for $L^2(\Lambda( M_K\sqcup N_K)$ and $H'_1$ stands for $L^2(\Lambda(M'_K\sqcup N'_K))$;
\item $H_2$ stands for $L^2(\Lambda(M_{[K,R]}\sqcup N_{[K,R]}))$ as well as $L^2(\Lambda(M'_{[K,R]}\sqcup N'_{[K,R]}))$;
\item $H_3$ stands for $L^2(\Lambda(M_{R\geq}\sqcup N_{R\geq}))$ as well as $L^2(\Lambda(M'_{R\geq}\sqcup N'_{R\geq}))$.
\end{enumerate}
With the above notation we have
\begin{enumerate}
\item $L^2(\Lambda(M_{\infty} \sqcup N_{\infty})) = H_{1}\oplus H_{2} \oplus H_{3} $;
\item $L^2(\Lambda(M'_{\infty}  \sqcup N'_{\infty})) = H'_{1} \oplus H_{2}\oplus H_{3} $;
\item $L^2(\Lambda((M_R\sqcup_{\partial M} M'_R) \sqcup (N_R \sqcup_{\partial N} N'_R)))=H_1 \oplus H_2 \oplus H_2 \oplus H'_1 $;
\item $L^2(\Lambda((M'_R\sqcup_{\partial M'} M'_R) \sqcup (N'_R\sqcup_{\partial N'} N'_R)))=H'_1 \oplus H_2 \oplus H_2 \oplus H'_1$.
\end{enumerate}

Choose a representative element of the $K$-theory class $\text{Ind}(\theta)$, denote it by $a$, with finite propagation.
Moreover, by section \ref{section Infinitesimal is a good zero} we can additional suppose $[1-a]$ is trivial on $ \tilde{X}\times [K,\infty)$.
Similarly, for $\text{Ind}(\theta')$ we have $a'$.

Define operators on $H_1 \oplus H_2 \oplus H_2 \oplus H'_1 $ and $H'_1 \oplus H_2 \oplus H_2 \oplus H'_1$ by
\[
A=
\begin{pmatrix}
a & 0 &0 & 0\\
0 & 0 & 0 & 0\\
0 & 0 & 0 & 0\\
0 & 0 & 0 & a'
\end{pmatrix}\begin{pmatrix}
H_1\\
H_2\\
H_2\\
H'_1
\end{pmatrix}\text{~and~}
A'=\begin{pmatrix}
a' & 0 &0 & 0\\
0 & 0 & 0 & 0\\
0 & 0 & 0 & 0\\
0 & 0 & 0 & a'
\end{pmatrix}\begin{pmatrix}
H'_1\\
H_2\\
H_2\\
H'_1
\end{pmatrix}.
\]	
Then, obviously we have that
\[
\text{Ind}(\theta)-\text{Ind}(\theta')=[A]-[A'].
\]

We claim that

\[[A]=\text{Ind}(\theta\sqcup \theta').\]
Let $R'=2(R-K)$.
In fact, we can assume that $R\geq 2K$ and $K$ is big enough, such that there is almost flat function $\phi$ and $\phi'$ satisfies
\begin{enumerate}
\item $\phi$ is constantly 1 on $M_{K}\sqcup M'_{K}$, constantly zero on $M_{R'\geq}\sqcup M'_{R'\geq}$;
\item $\phi'$ is constantly 1 on $N_{K}\sqcup N'_{K}$, constantly zero on $N_{R'\geq}\sqcup N'_{R'\geq}$;
\item Both $\phi$ and $\phi'$ are constant along boundary direction;
\item We have $\phi +\phi'=1$ on $H_1 \oplus H_2 \oplus H_2 \oplus H'_1 $.
\end{enumerate}
Thus we have
\begin{align*}
A=&\begin{pmatrix}
a & 0 &0 & 0\\
0 & 0 & 0 & 0\\
0 & 0 & 0 & 0\\
0 & 0 & 0 & 0
\end{pmatrix}\phi + \begin{pmatrix}
0 & 0 &0 & 0\\
0 & 0 & 0 & 0\\
0 & 0 & 0 & 0\\
0 & 0 & 0 & a'
\end{pmatrix}\phi'\\
=&
\sqrt{\phi}\begin{pmatrix}
a & 0 &0 & 0\\
0 & 0 & 0 & 0\\
0 & 0 & 0 & 0\\
0 & 0 & 0 & 0
\end{pmatrix}\sqrt{\phi}+\sqrt{\phi'}\begin{pmatrix}
0 & 0 &0 & 0\\
0 & 0 & 0 & 0\\
0 & 0 & 0 & 0\\
0 & 0 & 0 & a'
\end{pmatrix}\sqrt{\phi'}+\epsilon	
\end{align*}
here $\epsilon$ stands for an operator with sufficient small norm.
Actually, we can substitute the following matrixes
\[
\begin{pmatrix}
a & 0 &0 & 0\\
0 & 0 & 0 & 0\\
0 & 0 & 0 & 0\\
0 & 0 & 0 & 0
\end{pmatrix}\text{~and~}\begin{pmatrix}
0 & 0 &0 & 0\\
0 & 0 & 0 & 0\\
0 & 0 & 0 & 0\\
0 & 0 & 0 & a'
\end{pmatrix}
\]
by $\tilde{a}\in B(L^2(M_\infty\sqcup N_\infty))$ and $\tilde{a}'\in B(L^2(M'_\infty\sqcup N'_\infty))$ respectively.

Since
\[
\|(D+S_f)(D-S'_f)^{-1}-\tilde{a}\|\leq \eta,\;\|(D'+S'_f)(D'-S'_f)^{-1}-\tilde{a}'\|\leq \eta.
\]
Hence, we have that $[1-A]$ and $\sqrt{\phi}(D+S_f)(D-S'_f)^{-1}\sqrt{\phi}+ \sqrt{\phi'}(D'+S_{f'})(D-S'_{f'})^{-1}\sqrt{\phi'}$ give the same $K$-theory class.

Let $D_\sqcup$ be the signature operator on $(M\sqcup M') \sqcup (N\sqcup N')$.
We can define $S_{\sqcup,f}$ and $S'_{\sqcup,f}$ as
\[
S_{\sqcup,f}=S_f \sqcup S_{f'} \text{~and~} S'_{\sqcup,f}=S'_f \sqcup S'_{f'}
\]

Now one can see that
\begin{align*}
& \sqrt{\phi}(D_\sqcup+S_{\sqcup,f})(D_\sqcup-S'_{\sqcup,f})^{-1}\sqrt{\phi}+ \sqrt{\phi'}(D_\sqcup+S_{\sqcup,f})(D_\sqcup-S'_{\sqcup,f})^{-1}\sqrt{\phi'}\\
=& (D+S_f)(D-S'_f)^{-1}\phi+ (D'+S_{f'})(D'-S'_{f'})^{-1}\phi'+\epsilon,	
\end{align*}
where $\epsilon$ is an operator with sufficiently small norm.
Thus the Theorem follows.
\end{prf}

Now, to show that map $\text{Ind}: L_n (\pi_1 X)\to K_n(C^*(\tilde{X})^G)$ is well defined,
it remains only to recall that higher signature is bordism invariant.

%
%

Let $\theta$ be zero element in $L_n(\pi_1 X)$, recall Definition \ref{definition eL} or \cite[Definition 3.10]{WXY16}.
We take $\theta' =\{\partial_2 W,\partial M,\Phi_{\partial M}, \partial_2 V, \partial N ,\Psi|_{\partial_2 V},F|_{\partial_2 W}\}$, here various terms, such as $\partial_2 W$ and $\partial_2 V$, take from Definition \ref{definition eL}.
Obviously, $\theta'\in L_n(\pi_1 X)$, and by Theorem \ref{theorem homotopy trivial L}, we have $\text{Ind}(\theta')=0$.
Hence, the second item in the left side of Eq.(\ref{form difference formila}) is trivial in $K_n(C^*(\tilde{X})^G)\otimes \mathbb{Z}[\frac{1}{2}]$.
However, since higher signature is a bordism invariant \cite{H1991a}, two items in right side of Eq.(\ref{form difference formila}) are both trivial in $K_n(C^*(\tilde{X})^G)\otimes \mathbb{Z}[\frac{1}{2}]$.

Then from Eq.(\ref{form difference formila}) we can see that
\[
\text{Ind}: L_n(\pi_1 X)\to K_n(C^*(\tilde{X})^G)\otimes \mathbb{Z}[\frac{1}{2}],
\]
maps zero element to zero element. Thus $\text{Ind}$ is a well defined group homomorphism.

\begin{remark}\label{remark twist 2}
In a word, we have to consider $\mathbb{Z}[\frac{1}{2}]$ coefficient since the ``boundary of signature operator" is not signature, but twice the signature when the manifold is even dimensional.
\end{remark}

\subsection{A product formula}\label{subsection X times R}
To conclude this section, we present a product formula.
The proof is postpone to Appendix \ref{subsection X times R proof}.
As our proof is essentially the same as \cite[Appendix D]{WXY16}, we will go through details for even case only.

If $\theta=(M,\partial M,\phi, N, \partial N,\psi, f) \in L_n(\pi_1 X)$, let $ \theta \times  \mathbb{R}$ be the product of $\theta $ and $\mathbb{R}$, which defines an element in $L_{n+1}(\pi_1 X)$.
Here various undefined terms take the obvious meanings \cite{WXY16}.

Note that the construction in subsection \ref{subsection M infty and N infty} also applies to $\theta \times  \mathbb{R}$ and defines a $K$-theory class in $K_{n+1}(C^*(\tilde{X}\times  \mathbb{R} )^G)$.
And there is a natural homomorphism $\alpha : C^*(\tilde{X} )^G \otimes  C^*_L(\mathbb{R}) \to C^*(\tilde{X}\times  \mathbb{R} )^G$.

\begin{proposition}\label{proposition times proposition}
With the same notation as above, we have
\[
k_n \cdot \alpha_*\left(\text{Ind}(\theta)\otimes \text{Ind}_L(\mathbb{R})\right)=\text{Ind}(\theta \times  \mathbb{R})
\]
in $K_{n+1}(C^*(\tilde{X}\times \mathbb{R} )^G)$, where $\text{Ind}_L(\mathbb{R})$ is the $K$-homology class of the signature operator on $\mathbb{R}$ and $k_n=1$ if $n$ is even and $2$ if $n$ is odd.
\end{proposition}

\begin{remark}
We would like to remind readers that throughout this paper $\text{Ind}_L(\mathbb{R})$ is the $K$-homology class of the signature operator on $\mathbb{R}$.
Do NOT confuse with the map $\text{Ind}_L$ which we will defined in subsection \ref{section Infinitesimal is a good zero}.
\end{remark}

\section{Local index map for $\mathcal{N}_n(X)$}\label{Define for element of N}
In this section we address how to define the local index map $\text{Ind}_L: \mathcal{N}_n(X)\to K_n(C^*_L(\tilde{X})^G)$ and show it is well defined and admire additivity.

Before we started, we need a fair idea of a certain hybrid $C^*$-algebra, which introduced in \cite[Section 4.1]{WXY16}.
For the convenience of readers, we briefly recall some facts about hybrid $C^*$-algebras in subsection \ref{subsection CX}.

In subsection \ref{section Infinitesimal is a good zero}, using \emph{infinitesimal controlled homotopy equivalence}, we reduce the definition of local index map to the index map which we have carefully studied in section \ref{subsection Define for element of L}.

\subsection{$CX$ and hybrid $C^*$-algebra}\label{subsection CX}
Let $G$ be a countable discrete group, and let $X$ be proper metric space equipped with a proper $G$-action.

\begin{definition}[{\cite[Definition 4.1]{WXY16}}]
We define $C_c^*(X)^G$ to be the closed subalgebra of $C^*(X)^G$ generated by all elements $\alpha$ such that: for any $\epsilon>0$, there exists $G$-invariant cocompact set $K\subseteq Y$ such that the propagation of $\chi_{(X-K)}\alpha $ and $\alpha \chi_{(X-K)}$ are less than $\epsilon$.
\end{definition}

\begin{definition}[{\cite[Definition 4.2]{WXY16}}]
We define $C_{L,0,c}^*(X)^G$ to be the closed subalgebra of $C_{L,0}^*(X)^G$ generated by all elements $\alpha$ such that: for any $\epsilon>0$, there exists $G$-invariant cocompact set $K\subseteq Y$ such that the propagation of $\chi_{(X-K)}\alpha $ and $\alpha \chi_{(X-K)}$ are less than $\epsilon$.
\end{definition}

Let $X \times  [1, \infty)$ be the product space.
Let $\tilde{X}$ be the universal cover of $X$ and let $G = \pi_1 X$ be the fundamental group of $X$.
It is obvious that for any $r\geq 1$, the $C^*$ algebra $C^*_{L,0}(\tilde{X}\times [1,r];\tilde{X}\times [1,\infty))^G$ is a two sided ideal of $C^*_{L,0}(\tilde{X}\times [1,\infty))^G$.

\begin{proposition}
For $i=0,1$, we have
\begin{enumerate}
  \item $K_i(C^*_{L,0,c}(\tilde{X}\times [1,\infty))^G) \cong K_i(C_{L,0}^*(\tilde{X})^G)$ \cite[Proposition 4.4]{WXY16} ;
  \item $K_i(C^*_{L,c}(\tilde{X}\times [1,\infty))^G)= 0$ \cite[Lemma 4.6]{WXY16};
  \item $K_i(C_{c}^*(\tilde{X}\times [1,\infty))^G)\cong K_{i+1}(C^*_{L,0,c}(\tilde{X}\times [1, \infty))^G)$ \cite[Corollary 4.7]{WXY16}.
\end{enumerate}
\end{proposition}
%
%

We now introduce the space $CX$, which is defined as a rescaling of $X\times[1,\infty)$ along $X$ for any $t\in[1,\infty)$ such that
\begin{enumerate}
\item For $x\in X$ and $t_1,t_2 \in[1,\infty)$, we have $d((x,t_1), (x, t_2))= |t_2-t_1 |$;
\item For $x_1, x_2 \in X$ and $t \in[1,\infty)$, we have $d((x_1,t), (x_2, t))$ is no decreasingly tends to infinity as $t$ tends to infinity.
\end{enumerate}

\begin{remark}
If $M$ is a closed Riemannian manifold, in $M\times[1,\infty)$ we equipped with a (complete) Riemannian metric of the form
\[
ds^2 =dt^2 +t^2g_{ij}dy^idy^j
\]
where $t$ is the coordinate on $[1,\infty)$, $g_{ij}$ is a Riemannian metric on $M$.
This metric will satisfied above conditions.
\end{remark}

Note that there is a proper continuous map
\[
\tau : CX\to  X \times [1, \infty),
\]
define by $\tau(x, t) = (x, t)$ which induces a $C^*$-algebra homomorphism
\[
 \tau_* : C^*(CX) \to C^*_c (X \times  [1,\infty)).
\]
Similarly, we have
\[
 \tau_* : C^*_{L,0} (CX) \to C^*_{L,0,c}  (X \times  [1,\infty)).
\]
There are also obvious $G$-equivariant versions of theory of hybrid $C^*$ algebra.

\subsection{A construction on infinitesimal controlled homotopy equivalence}\label{section Infinitesimal is a good zero}

Consider $\theta = (M,\partial M, \phi,N,\partial N,\psi,f) \in \mathcal{N}_n(X)$ (\cite[Definition 3.13]{WXY16} or Definition \ref{definition N}).
$CM, CN$ are obtained in the same way as $CX$.

For each $s\in [1,\infty)$ denote by $X_s$ the sub-manifold $X\times \{s\} \hookrightarrow CX$.
Do not confuse it with $X$ itself.
Moreover, for $m\in \mathbb{N}$, we should also differ $C^*(\sqcup_m X_m)$ from  $C^*(\sqcup_m X)$, since $\sqcup_m X_m$ is not coarsely equivalent to $\sqcup_m X$.
Similarly, we denote $M_s$ (resp. $N_s$) the sub-manifold $M\times \{s\} \hookrightarrow CM$ (resp. $N\times \{s\} \hookrightarrow CN$).  We also have $C^*(\sqcup_m M_m)$ (resp. $C^*(\sqcup_m N_m)$).

Let $\tilde{X}$ be the universal cover of $X$, then $\tilde{X}$ is a regular $G$-covering of $X$ where $G=\pi_1 X$ is the fundamental group of $X$.
Denote by $C\tilde{M}$ (resp. $C\tilde{N}$) the pull back $G$-covering induced by $\phi :M\to X$ (resp. $\psi :N\to X$).

Hence, for each $t\in[0,1]$ we have $\sqcup_m \tilde{M}_{m+t}$ and $\sqcup_m \tilde{N}_{m+t}$.
Let $d_{\sqcup_m \tilde{M}_{m+t}}$ (resp. $d_{\sqcup_m \tilde{N}_{m+t}}$) be the ($G$-equivariant) \emph{de Rham differential operator} on $\sqcup_m \tilde{M}_{m+t}$ (resp. $\sqcup_m \tilde{N}_{m+t}$).
Let $D_{\sqcup_m \tilde{M}_{m+t}}$ be $d_{\sqcup_m \tilde{M}_{m+t}}+d_{\sqcup_m \tilde{M}_{m+t}}^*$ and $D_{\sqcup_m \tilde{N}_{m+t}}$ be $d_{\sqcup_m \tilde{N}_{m+t}}+d_{\sqcup_m \tilde{N}_{m+t}}^*$.

Note that for $t\in[0,1]$, the ``boundary" of $\sqcup_m \tilde{M}_{m+t}$ (resp. $\sqcup_m \tilde{N}_{m+t}$) is $\sqcup_m \tilde{\partial M}_{m+t}$ (resp. $\sqcup_m \tilde{\partial N}_{m+t}$).
Since $f$ restricts to the boundary is an \emph{infinitesimally controlled homotopy equivalence} over $X$, $f$ induces a controlled homotopy equivalence between $\sqcup_m \tilde{\partial M}_{m+t}$ and $\sqcup_m \tilde{\partial N}_{m+t}$) over $\sqcup_m \tilde{X}_{m+t}$.

We set
\[
D_t:=\begin{pmatrix}
D_{\sqcup_m \tilde{M}_{m+t}} & 0\\
0 & D_{\sqcup_m \tilde{N}_{m+t}} 	
\end{pmatrix}\text{~and~}S_t:=\begin{pmatrix}
S_{\sqcup_m \tilde{M}_{m+t}} & 0\\
0 & -S_{\sqcup_m \tilde{N}_{m+t}}	
\end{pmatrix}.
\]

By subsection \ref{subsection M infty and N infty}, for any $t\in[0,1]$, we can define a relative index $\text{Ind}(D_t)$ in $K_n(C^*(\sqcup_m \tilde{X}_{m+t})^G)$.
Although we do not known whether there exist a path connected $\text{Ind}(D_t)$ and the trivial element in $K_n(C^*(\sqcup_m X_{m+t})^G)$, we can define a $K$-theory class in $K_n(C_L^*(\tilde{X})^G)$.
We denote this $K$-theory class by $\text{Ind}_L(\theta)$.

\begin{remark}
In \cite{Y97} Yu proved that $K_*(C_L^*(\tilde{X})^G)$ is isomorphic to $K$-homology $K_*(X)$ of $X$.
Let $M$ and $N$ be two compact manifolds, and let $G$ be the fundamental group of $X$.
Consider $\theta = (M,\emptyset , \phi,N,\emptyset ,\psi,f)\in \mathcal{N}_n(X)$.
Under the isomorphism constructed in \cite{Y97}.
It is not hard to see that $\text{Ind}_L(\theta)=\phi_*([D_{\tilde{M}}])-\psi_*([D_{\tilde{N}}])$, where $[D_{\tilde{M}}]$ (resp. $[D_{\tilde{N}}]$) is the $K$-homology class of signature operator on $\tilde{M}$ (resp. $\tilde{N}$).
\end{remark}

Now, we will show that the map $\text{Ind}_L: \mathcal{N}_n(X)\to K_n(C^*_L(\tilde{X})^G)$ is well defined.
Analogous to the Theorem \ref{theorem homotopy trivial L} and Theorem \ref{theorem L differ} in subsection \ref{subsection Define for element of L}, we have the following two theorems.

\begin{theorem}\label{thm.infinitesimal homotopy trivial N}
Suppose $\theta=(M,\partial M,\phi, N, \partial N,\psi, f) \in \mathcal{N}_n(X)$.
And $f: M \to N$ is an \emph{infinitesimally controlled homotopy equivalence} over $X$, then $\text{Ind}_L(\theta)=0$.
\end{theorem}

Recall, the addition of the abelian group $\mathcal{N}_n(X)$ is given by disjoint union(\cite[Definition 3.14]{WXY16}, Definition \ref{definition eN}).
Thus for any two elements $\theta,\theta'\in\mathcal{N}_n(X)$, we have $\theta\sqcup \theta'$,  the disjoint union of $\theta$ and $-\theta'$, belongs to $\mathcal{N}_n(X)$.

\begin{theorem}\label{theorem N differ}
Let $\theta, \theta' \in \mathcal{N}_n(X)$. 
We further suppose $\theta =(M, \partial M, \phi, N,\partial N,\psi, f)$, $\theta'=(M', \partial M', \phi', N',\partial N', \psi', f')$ and satisfies the following conditions:
\begin{enumerate}
\item $\partial M=\partial M'$,\;$\partial N=\partial N'$, restricting to the boundary $f=f'$,\;$\phi=\phi'$ and $\psi=\psi'$;
\item There exist two manifolds with boundary $(W, \partial W)$ and $(V, \partial V)$, continuous maps $\Phi:W \to X$(resp. $\Psi:V \to X$).
Manifolds $W$ and $V$ both of dimension $n+1$.
Moreover, $\partial W = M\sqcup_{\partial M} M'$(resp. $\partial V = N\sqcup_{\partial N} N'$) and $\Phi$(resp. $\Psi$) restricts to $\phi\sqcup \phi'$(resp. $\psi\sqcup \psi'$) on $M\sqcup_{\partial M} M'$(resp. $N\sqcup_{\partial N} N'$);
\item There exist a degree one normal map $F:W\to V$ such that $\Psi \circ F=\Phi$.
Moreover, $F$ restricts to $f\sqcup f'$ on $M\sqcup_{\partial M} M'$.
\end{enumerate}
we have
\begin{equation}\label{form N differ}
\text{Ind}_L(\theta)-\text{Ind}_L(\theta')=\text{Ind}_L(\theta\sqcup \theta')-\text{Ind}_L(\theta'\sqcup\theta').
\end{equation}
\end{theorem}

For the same reason as we mentioned in Remark \ref{remark twist 2}, we have
\[
\text{Ind}_L: \mathcal{N}_n(X)\to K_n(C_L^*(\tilde{X})^G)\otimes \mathbb{Z}[\frac{1}{2}],
\]
is a well defined group homomorphism.

\section{Higher rho invariant map for $L_n(\pi_1 X,X)$}\label{section Mapping surgery to analysis}
Let $X$ be a closed oriented topological manifold of dimension $n\geq 5$. 
With the technical and results of Section \ref{subsection Define for element of L} and Section \ref{Define for element of N}, we are now in a position to define ``Mapping surgery to analysis" and to prove the additivity of higher rho map (Eq.(\ref{eq.higherrhomap}) and Theorem \ref{theorem mapping sur t ana}).

For our purposes, we need to seek help from $L_{n+1} (\pi_1 X, X)$.
Hence, before defining the higher rho map, it is helpful to investigate the index map $\hat{\rho} :L_{n+1}(\pi_1 X, X)\to K_{n+1}(C^*_c(\tilde{X}\times [1,\infty)^G)$ in the coming subsection.

\subsection{Index of $L_n(\pi_1 X,X)$}
\label{subsection Define for element of Lpi}

For $\theta = (M,\partial_{\pm}M,\phi ,N,\partial_{\pm} N,\psi,f)\in L_{n+1}(\pi_1 X,X)$ (\cite[Definition 3.19]{WXY16} or Definition \ref{definition Ln}), let $C_-M$ and $C_-N$ be $M \sqcup_{\partial_- M} C\partial_- M$ and $N \sqcup_{\partial_- N} C\partial_- N$ respectively.

 $C_-M$ (resp. $C_-N$) is manifold with boundary $\partial C_-M =\partial_+ M \sqcup_{\partial\partial_- M} C\partial\partial_- M$ (resp. $\partial C_-N =\partial_+ N \sqcup_{\partial\partial_- N} C\partial\partial_- N$).
Since $f:\partial\partial_+ M\to \partial\partial_+ N$ is an \emph{infinitesimally controlled homotopy equivalence} over $X$, it induces a  controlled homotopy equivalence restricting to $\partial C_-M$. Hence, by subsection \ref{subsection M infty and N infty}, this defines a $K$-theory class in $K_{n+1}(C^*(C\tilde{X})^G)$.
Then by rescaling, we actually obtain a class in $K_{n+1}(C_c^*(\tilde{X}\times [1, \infty))^G)$.
This class is written as $\hat{\rho} (\theta)$.

Let $\theta$ and $\theta'$ be two elements in $L_{n+1}(\pi_1 X, X)$.
Since the addition of the abelian group $L_{n+1}(\pi_1 X, X)$ is given by disjoint union(\cite[Definition 3.20]{WXY16} or Definition \ref{definition eLn}).
We have $\theta\sqcup \theta'$, the disjoint union of $\theta$ and $-\theta'$, belongs to $L_{n+1}(\pi_1 X, X)$.
Analogous to the situation for index map for $L$-groups (Theorem \ref{theorem L differ}), we have the following theorem.
The proof is essentially the same as Theorem \ref{theorem L differ}, so we omit it.

\begin{theorem}\label{theorem Lpi differ}
Let $\theta$ and $\theta'$ be two elements in $L_{n+1}(\pi_1 X, X)$. 
We further suppose $\theta =(M, \partial_{\pm} M,\phi, N,\partial_{\pm} N, \psi, f )$ and $\theta' = (M', \partial_{\pm} M', \phi', N',\partial_{\pm} N',\psi',f')$ such that $\partial_+ M=\partial_+ M'$, $\partial_+ N=\partial_+ N'$ and restricting to the $\partial_+$ boundary $f=f'$ and $\phi=\phi'$.
Then we have
\[
\hat\rho(\theta)-\hat\rho(\theta')=\hat\rho(\theta \sqcup \theta')-\hat\rho(\theta' \sqcup\theta').
\]
\end{theorem}

By Theorem \ref{theorem Lpi differ} we can see that $\hat{\rho} :L_{n+1}(\pi_1 X, X) \to  K_{n+1}(C_c^*(\tilde{X}\times [1, \infty))^G)$ is well defined.
We now come to show that it is a well defined group homomorphism.
The point is to show that $\hat\rho(\theta\sqcup \theta')$ is zero.

Although we can not use bordism invariance to show that $\hat\rho(\theta\sqcup \theta')$ is zero, we can show this directly.
Since $\hat\rho(\theta\sqcup \theta')$ lies in the image of the following map:
\[
\begin{tikzcd}
K_*^G(C\tilde{X}) \ar[r] \ar[d] &K_*(C^*(C\tilde{X})^G) \ar[d]\\
K_*(C^*_{L,c}(\tilde{X}\times [1,\infty))^G) \ar[r]& K_*(C^*_c(\tilde{X}\times [1,\infty))^G) 	
\end{tikzcd}
\]
where $K_*^G(C\tilde{X})$ is the $G$-equivariant $K$-homology of $C\tilde{X}$.
By \cite[Lemma 4.6]{WXY16}, $K_*(C^*_{L,c}(\tilde{X}\times [1,\infty)^G) $ is always trivial.
Thus $\hat\rho(\theta\sqcup \theta')$ is always zero in $K_*(C^*_c(\tilde{X}\times [1,\infty)^G)$.

Hence, we have
\[
\hat{\rho} :L_*(\pi_1 X, X)\to K_*(C^*_c(\tilde{X}\times [1,\infty)^G)
\]
is a well defined group homomorphism.

\subsection{Higher rho map for $\mathcal{S}_n(X)$}\label{subsection Define for element of S}
Suppose $\theta=(M,\partial M, \phi, N,\partial N, \psi, f) \in \mathcal{S}_n(X)$ (\cite[Definition 3.4]{WXY16} or Definition \ref{definition S}).
We use the same notations as in subsection \ref{Define for element of N}.
For $t\in[0,1]$ and $m\in\mathbb{N}$, we have $\sqcup_m \tilde{M}_{m+t}$ and $\sqcup_m \tilde{N}_{m+t}$.
And $S_{\sqcup_m \tilde{M}_{m+t}}$ (resp. $S_{\sqcup_m \tilde{N}_{m+t}}$), the \emph{Poincar\'e duality operator} of $d_{\sqcup_m \tilde{M}_{m+t}}$ (resp. $d_{\sqcup_m \tilde{N}_{m+t}}$).
As well as, $D_t$ and $S_t$.
By Section \ref{Define for element of N}, we can define an element $\alpha \in K_n(C^*_L(\tilde{X})^G)\otimes \mathbb{Z}[\frac{1}{2}]$.

When $t=0$, $\tilde{M}\times \{1\} \hookrightarrow \sqcup_m \tilde{M}_{m}$ (resp.  $\tilde{N}\times \{1\}\hookrightarrow \sqcup_m \tilde{N}_{m}$), $d_{\sqcup_m \tilde{M}_{m}}|_{\tilde{M}}=d_{\tilde{M}}$ (resp. $d_{\sqcup_m \tilde{N}_{m}}|_{ \tilde{N}}=d_{\tilde{N}}$), and $S_{\sqcup_m \tilde{M}_{m}}|_{\tilde{M}}=S_{\tilde{M}}$ (resp. $S_{\sqcup_m \tilde{N}_{m}}|_{\tilde{N}}=S_{\tilde{N}}$).

Moreover, since $f:M\to N$ is a homotopy equivalence over $X$.
We can connect $\alpha(1)$ to the trivial element in $K_n(C^*(\tilde{X})^G)$.
Hence, we have the following map
\[
\rho: \mathcal{S}_n(X)\to K_n(C^*_{L,0}(\tilde{X})^G).
\]

We want to show the map
\begin{equation}\label{eq.higherrhomap}
k_n\cdot\rho: \mathcal{S}_n(X)\to K_n(C^*_{L,0}(\tilde{X})^G)\otimes \mathbb{Z}[\frac{1}{2}],	
\end{equation}
is a well defined group homomorphism, where $k_n=1$ if $n$ is even and $2$ if $n$ is odd.
This is a corollary of the following theorem.
Recall that  we have $\mathcal{S}_n(X)\cong L_{n+1} (\pi_1 X, X)$ (\cite[Section 3.3]{WXY16} or see Appendix \ref{section WXY surgery sequence}).

\begin{theorem}[{\cite[Theorem 6.9]{WXY16}}]\label{theorem pre well rho}
Let $X$ be a closed oriented topological manifold of dimension $\geq 5$.
We have the following commutative diagram of $K$-groups

\medskip
\begin{adjustbox}{center}
\begin{tikzcd}
L_{n+1}(\pi_1 X,X) \arrow{r}{\hat{\rho}} &K_{n+1} (C^*_c (\tilde{X} \times  [1,\infty) )^G )\arrow{d}{\partial_*}\otimes \mathbb{Z}[\frac{1}{2}] \\
\mathcal{S}_n(X)\arrow[u,"c_*"'] \arrow{r}{k_n \cdot \rho} & K_n( C^*_{L,0,c} (\tilde{X} \times [1, \infty))^G)\otimes \mathbb{Z}[\frac{1}{2}] = K_n(C^* _{L,0}(\tilde{X})^G)\otimes \mathbb{Z}[\frac{1}{2}]
\end{tikzcd}	
\end{adjustbox}
where $\partial_*$ is the connecting map in the $K$-theory long exact sequence associated to the following short exact sequence of $C^*$-algebras
$$ 0 \to  C^*_{L,0,c} (\tilde{X} \times [1, \infty))^G \to C^*_{L,c} (\tilde{X} \times [1, \infty))^{G} \to C^*_c (\tilde{X} \times [1, \infty))^{G} \to 0,$$
and $k_n = 1$ if $n$ is even and $2$ if $n$ is odd.
\end{theorem}

The strategy to prove this theorem is to compute the map $\partial_*$ in a Mayer-Vietoris sequence.

Before we go into details, we specify which Mayer-Vietoris sequence we are going to use.
Let $\mathcal{A}:=C^*_{L,0}(\tilde{X}\times \mathbb{R})^G$.
Define a bunch of $C^*$ algebras as
\begin{align*}
&\mathcal{A}_- = \bigcup_{n\in \mathbb{N}} C^*_{L,0}(\tilde{X}\times [-\infty, n]; \tilde{X}\times \mathbb{R})^G\\
&\mathcal{A}_+ = \bigcup_{n\in \mathbb{N}} C^*_{L,0}(\tilde{X}\times [-n, \infty]; \tilde{X}\times \mathbb{R})^G\\
&\mathcal{A}_\cap = \bigcup_{n\in \mathbb{N}} C^*_{L,0}(\tilde{X}\times [-n, n]; \tilde{X}\times \mathbb{R})^G,	
\end{align*}
all of which are two sided ideal of $\mathcal{A}$.
Also, we have
\[
\mathcal{A}_- + \mathcal{A}_+ = \mathcal{A},\;\mathcal{A}_- \cap \mathcal{A}_+ =\mathcal{A}_\cap.
\]
Thus there is Mayer-Vietoris sequence
\[
\begin{tikzcd}
K_0(\mathcal{A}_\cap)\arrow{r} &K_0(\mathcal{A}_+)\otimes K_0(\mathcal{A}_-) \arrow{r}& K_0(\mathcal{A})\arrow{d}{\partial_{MV}} \\
K_1(\mathcal{A})\arrow[u,"\partial_{MV}"'] & K_1(\mathcal{A}_+)\otimes K_0(\mathcal{A}_-) \arrow{l} & K_0(\mathcal{A}_\cap)\arrow{l}.	
\end{tikzcd}
\]
Similarly, let $\mathcal{B}:=C^*_L(\mathbb{R})$, we define
\begin{align*}
&\mathcal{B}_- = \bigcup_{n\in \mathbb{N}} C^*_{L}( [-\infty, n];  \mathbb{R})\\
&\mathcal{B}_+ = \bigcup_{n\in \mathbb{N}} C^*_{L}( [-n, \infty];  \mathbb{R})\\
&\mathcal{B}_\cap = \bigcup_{n\in \mathbb{N}} C^*_{L}( [-n, n]; \mathbb{R}),
\end{align*}
and again, we have Mayer-Vietoris sequence
\[
\begin{tikzcd}
K_0(\mathcal{B}_\cap)\arrow{r} &K_0(\mathcal{B}_+)\otimes K_0(\mathcal{B}_-) \arrow{r}& K_0(\mathcal{B})\arrow{d}{\partial_{MV}} \\
K_1(\mathcal{B})\arrow[u,"\partial_{MV}"'] & K_1(\mathcal{B}_+)\otimes K_0(\mathcal{B}_-) \arrow{l} & K_0(\mathcal{B}_\cap)\arrow{l}.	
\end{tikzcd}
\]

\begin{prf}
To compute $\partial_* (\hat{\rho}(c_*(\theta)))$, we explicitly construct the lifting of $\hat{\rho}(c_*(\theta))$ in $C^*_{L,c}(\tilde{X}\times[1, \infty])^G$.

Let $a_\theta(n)$ be $\chi_n \hat{\rho}(c_*(\theta))\chi_n$, where $\chi_n$ the characteristic function on $\tilde{X}\times[n,\infty]$.
Then $a_\theta(t)=(n+1-t)a_\theta(n)+(t-n)a_\theta(n+1)\in C^*_{L,c}(\tilde{X}\times[1, \infty])^G$ is a lift of $\hat{\rho}(c_*(\theta))$.

On the other hand, one can see $a_\theta$ is also a lift of $\rho'(\theta\times \mathbb{R})$ under the homomorphism
\[
\partial_{MV}: K_{n+1}(C^*_{L,0}(\tilde{X}\times\mathbb{R})^G)\to K_n(C^*_{L,0}(\tilde{X})^G).
\]

Thus, we have
\[
\partial_* (\hat{\rho}(c_*(\theta)))=\partial_{MV} (\rho'(\theta\times \mathbb{R})).
\]

However, the right side one can be compute by the following formalism argument.

Note that there is a natural homomorphism $\alpha: C^*_{L,0}(\tilde{X})\times \mathcal{B}\to \mathcal{A}$, which restrict to homomorphisms
\[
\alpha: C^*_{L,0}(\tilde{X})^G\times \mathcal{B}_\pm\to \mathcal{A}_\pm, \ \alpha: C^*_{L,0}(\tilde{X})^G\times \mathcal{B}_\cap\to \mathcal{A}_\cap.
\]

And we have the following commutative diagram

\medskip
\begin{adjustbox}{center}
\begin{tikzcd}[column sep=tiny]
K_n(C^*_{L,0}(\tilde{X})^G)\otimes K_1(\mathcal{B})\ar[r]\ar[d,"1\times \partial_{MV}"] &K_{n+1}(C^*_{L,0}(\tilde{X})^G\otimes \mathcal{B}) \ar[r] & K_{n+1}(C^*_{L,0}(\tilde{X}\times\mathbb{R})^G) \ar[d, "\partial_{MV}"] \\
K_n(C^*_{L,0}(\tilde{X})^G)\otimes K_0(\mathcal{B}_\cap)\ar[r] & K_{n}(C^*_{L,0}(\tilde{X})^G\otimes \mathcal{B}_\cap) \ar[r] & K_{n}(\mathcal{A}_\cap)=K_{n}(C^*_{L,0}(\tilde{X})^G).
\end{tikzcd}
\end{adjustbox}

Using the main result of \cite[Theorem 3.2]{Y97}, we know that the local index map from $K_i (\mathbb{R})$ to $K_i (C^*_L (\mathbb{R} ))$ is an isomorphism, where $K_i (\mathbb{R} )$ denote the $K$-homology of $\mathbb{R}$.

Now the proof is completed by the a product formula(Proposition \ref{prop.times propositionfors} or \cite[Theorem 6.8]{WXY16}).
\end{prf}

Given $\theta=(M,\partial M,\phi, N, \partial N,\psi, f) \in \mathcal{S}_n(X)$, let $\theta \times  \mathbb{R}$ be the product of $\theta $ and $\mathbb{R}$, which defines an element in $\mathcal{S}_{n+1}(X\times \mathbb{R})$.
Here various undefined terms take the obvious meanings \cite{WXY16}.

Note that $\theta \times  \mathbb{R}$ defines a $K$-theory class in $K_{n+1}(C^*_{L,0}(\tilde{X}\times  \mathbb{R} )^G)$.
Also there is a natural homomorphism $\alpha : C^*_{L,0}(\tilde{X} )^G \otimes  C^*_L(\mathbb{R}) \to C^*_{L,0}(\tilde{X}\times  \mathbb{R} )^G$.
Now we have the following product formula, which essentially the same as \cite[Theorem 6.8]{WXY16}.

\begin{proposition}[{\cite[Theorem 6.8]{WXY16}}]\label{prop.times propositionfors}
With the notations as above, we have
\[
k_n \cdot \alpha_*(\rho(\theta)\otimes \text{Ind}_L(\mathbb{R}))=\rho(\theta \otimes \mathbb{R})
\]
in $K_{n+1}(C^*(\tilde{X}\times  \mathbb{R} )^G)$, where $\text{Ind}_L(\mathbb{R})$ is the $K$-homology class of the signature operator on $\mathbb{R}$ and $k_n=1$ if $n$ is even and $2$ if $n$ is odd.
\end{proposition}

\begin{prf}
Use the same notations as before.
We need to consider $\theta \times  \mathbb{R}\in \mathcal{S}_{n+1}(X\times \mathbb{R})$.
Note that, for $s\in[1,\infty)$, $(\tilde{M}\times \mathbb{R})_s \neq \tilde{M}_s\times \mathbb{R}$ (resp. $(\tilde{N}\times \mathbb{R})_s \neq \tilde{N}_s\times \mathbb{R}$).
As $(\tilde{M}\times \mathbb{R})_s$ (resp. $(\tilde{N}\times \mathbb{R})_s $) is obtained from $\tilde{M}\times \mathbb{R}$ (resp. $\tilde{N}\times \mathbb{R}$) by rescaling $\tilde{M}\times \mathbb{R}$ (resp. $\tilde{N}\times \mathbb{R}$) according to $s$.
Moreover, we use product metric on $\tilde{M}\times \mathbb{R}$ (resp. $\tilde{N}\times \mathbb{R}$).
Nevertheless, we can choose the $K$-homology class of the signature operator on $\mathbb{R}$ with arbitrary small propagation \cite{Y97}.

In a word, in order to compute $\rho(\theta \otimes \mathbb{R})$, we can replace $\sqcup_m (\tilde{M}\times \mathbb{R})_{m+t}$ (resp. $\sqcup_m (\tilde{N}\times \mathbb{R})_{m+t}$) by $(\sqcup_m \tilde{M}_{m+t})\times \mathbb{R}$ (resp. $(\sqcup_m \tilde{N}_{m+t})\times \mathbb{R}$).

For $t\in[0,1]$, the $K$-theory class in $K_n(C^*(\sqcup_m \tilde{X}_{m+t})^G)$, which we get from $\sqcup_m \tilde{M}_{m+t}$ and $\sqcup_m \tilde{N}_{m+t}$, is the relative index.
So satisfied the product formula of Proposition \ref{proposition times proposition}.

Moreover, by the proof of Proposition \ref{proposition times proposition} we have a path connect $\alpha(1)$ and trivial element in $K_n(C^*(\tilde{X})^G)$.
It is no hard to see that this path also satisfied the product formula of Proposition \ref{proposition times proposition}.
\end{prf}
%
%
%
\begin{remark}
The $k_n$ appears here and in the argument of bordism invariance of signature operator is the reason why we need to consider $K$-theory with coefficient $\mathbb{Z}[\frac{1}{2}]$ coefficient.
\end{remark}

\subsection{Commutative diagram}\label{subsection Commutative diagram}
In this subsection, we prove the following theorem.
\begin{theorem}\label{theorem mapping sur t ana}
Let $X$ be a closed oriented topological manifold of dimension $\geq 5$.
We have commutative diagram of abelian groups

\medskip
\begin{adjustbox}{center}
\begin{tikzcd}[column sep=tiny]
\mathcal{N}_{n+1}(X)\ar[r]\ar[d, "\text{Ind}_L"]&L_{n+1}(\pi_1 X)\ar[r]\arrow[d, "\text{Ind}"]&\mathcal{S}_n(X)\ar[r]\ar[d, "k_n\cdot\rho"]& \mathcal{N}_n(X)\ar[d, "k_n\cdot\text{Ind}_L"]\\
K_{n+1}(C^*_L(\tilde{X})^G)\otimes \mathbb{Z}[\frac{1}{2}]\ar[r]&K_{n+1}(C^*(\tilde{X})^G)\otimes \mathbb{Z}[\frac{1}{2}]\ar[r]&K_n(C^*_{L,0}(\tilde{X})^G)\otimes \mathbb{Z}[\frac{1}{2}]\ar[r]&K_n(C^*_L(\tilde{X})^G)\otimes \mathbb{Z}[\frac{1}{2}]
\end{tikzcd}	
\end{adjustbox}
where, $G=\pi_1 X$ is the fundamental group of $X$, $\tilde{X}$ is the universal covering of $X$, and $k_n=1$ if $n$ is even and $2$ if $n$ is odd.
\end{theorem}

\begin{prf}
The commutativity of the left and the right square follows immediately from definition.
Now, we focus on the middle square
\[
\begin{tikzcd}
L_{n+1}(\pi_1 X)\ar[r]\arrow[d, "\text{Ind}"]&\mathcal{S}_n(X)\ar[d, "k_n\cdot\rho"]\\
K_{n+1}(C^*(\tilde{X})^G)\otimes \mathbb{Z}[\frac{1}{2}]\ar[r]&K_n(C^*_{L,0}(\tilde{X})^G)\otimes \mathbb{Z}[\frac{1}{2}].
\end{tikzcd}
\]

The commutativity of the above square can implied by following diagram

\medskip
\begin{adjustbox}{center}
\begin{tikzcd}[column sep=small]
&\mathcal{S}_n(X) \arrow[d, "c_*"] \arrow[ddr, "k_n\cdot\rho"] &\\
L_{n+1}(\pi_1X)\arrow[ur]\arrow[r, "j_*"]\arrow[d,"\text{Ind}"]& L_{n+1}(\pi_1 X,X)\arrow[d,"\hat{\rho}"]& \\
K_{n+1}(C^*(\tilde{X})^G)\otimes \mathbb{Z}[\frac{1}{2}]\arrow[r]&K_{n+1}(C_c^*(\tilde{X}\times [1,\infty))^G)\otimes \mathbb{Z}[\frac{1}{2}]\arrow[r, "\partial_*"] & K_n(C_{L,0}^*(\tilde{X})^G)\otimes \mathbb{Z}[\frac{1}{2}].
\end{tikzcd}
\end{adjustbox}

\end{prf}

\begin{remark}
We actually can eliminate the $\mathbb{Z}[\frac{1}{2}]$ coefficient, by the isomorphism of de Rham complex and simplicial complex.
However, we do not know how to eliminate it with a pure de Rham argument until now, so we will leave the $\mathbb{Z}[\frac{1}{2}]$ alone.
\end{remark}

In the last we mention that our definition coincide with the one in \cite{WXY16}.
This is actually immediate since both of them can be defined by hybrid $C^*$-algebra, and these hybrid algebra definition coincide with each other obviously.

\appendix
\section{Surgery long exact sequence}\label{section WXY surgery sequence}

In this appendix we list the definitions and main results of \cite[Section 3]{WXY16}.
Throughout, manifolds are suppose to be oriented and maps are suppose to be oriented-preserving.
For simplicity, we will suppress the terminology orientation in the context.

Let $X$ be a topological manifold.
Fix a metric on $X$ which is compatible with the topology of $X$.
If $X$ is a manifold with boundary, we  denote the boundary of $X$ by $\partial X$.
We begin with the notion of \emph{infinitesimal controlled homotopy equivalence}.

\begin{definition}
Let $M$ and $N$ be two compact Hausdorff spaces equipped with continuous maps $\phi: M \to X$ and $\psi: N \to X$.
A continuous map $f : M \to N$ is said to be an \emph{infinitesimally controlled homotopy equivalence} over $X$, if there exist proper continuous maps
\[
\begin{tikzcd}[column sep=small]
M       \times        [1,\infty) \arrow[rr, "F", yshift=0.7ex] \arrow[dr,"\Phi"'] & & N       \times        [1,\infty)\arrow[ll, "G", yshift=-0.7ex]\arrow[dl,"\Psi"]\\
& X        \times        [1,\infty)  &
\end{tikzcd}
\]
such that
\begin{enumerate}
\item $\Psi \circ F = \Phi$;
\item $F|_{M        \times        \{1\}}=f$, $\Phi|_{M  \times  \{1\}}=\phi$, $\Psi|_{N        \times        \{1\}}=\psi$;
\item There exists a proper continuous homotopy $\{H_s \}_{0\leq s \leq 1}$ between
\[
H_0 = F \circ G \text{~and~} H_1 = Id: N \times [1,\infty) \to N \times  [1,\infty)
\]
such that the diameter of the set $\{\Phi (H_s (z,t))~|~0 \leq s \leq1\}$ goes uniformly (i.e. independent of $z \in N$) to zero, as $t \to \infty $;
\item There exists a proper continuous homotopy $\{H'_s \}_{0\leq s \leq 1}$ between
\[
H'_0 = G \circ F \text{ and } H'_1 = Id: M \times [1,\infty) \to M  \times        [1,\infty)
\]
such that the diameter of the set $ \{\Psi (H'_s (y,t))~|~ 0 \leq s \leq1\}$ goes uniformly (i.e. independent of $y        \in M$) to zero, as $t \to \infty $.
\end{enumerate}
\end{definition}

We will also need the following notion of restrictions of homotopy equivalences gaining infinitesimal control on parts of spaces.

Suppose $(M,\partial M)$ and $(N,  \partial N)$ are two manifolds with boundary equipped with continuous maps $\phi: M \to  X$ and $\psi: N \to  X$.
Let $f : (M, \partial M) \to (N,  \partial N)$ be a homotopy equivalence such that $\psi \circ f =  \phi$.
Let $g:  (N,  \partial N)\to (M, \partial M)$ be a homotopy inverse of $f$.
Note that $\phi \circ g \neq \psi$ in general.
Let $\{h_s \}_{ 0 \leq s\leq 1}$ be a homotopy between $f \circ g$ and $Id: (N,  \partial N) \to  (N,  \partial N)$.
Similarly, let $\{h'_s \}_{ 0 \leq s\leq 1}$ be a homotopy between $ g \circ f$ and $Id: (M,\partial M) \to  (M,\partial M)$.

Since we have use $CM$ for other purpose, we denote $M \sqcup_{\partial M} \partial M \times  [1, \infty)$ by $M_\infty$ instead.

\begin{definition}[{\cite[Definition 3.3]{WXY16}}]
With the above notations, we say that on the boundary $f$ restricts to an \emph{infinitesimally controlled homotopy equivalence} $f|_{\partial M} : \partial M \to \partial N$ over $X$, if there exist proper continuous maps
\[
\begin{tikzcd}[column sep=small]
M_\infty \arrow[rr, "F", yshift=0.7ex] \arrow[dr,"\Phi"'] & & N_\infty \arrow[ll, "G", yshift=-0.7ex]\arrow[dl,"\Psi"]\\
& X        \times        [1,\infty)  &
\end{tikzcd}
\]
such that
\begin{enumerate}
\item $\Psi \circ F = \Phi$;
\item $F|_M=f$, $\Phi|_M =\phi$, $\Psi|_N=\psi$;
\item There exists a proper continuous homotopy $\{H_s \}_{0\leq s \leq 1}$ between
\[
H_0 = F \circ G \text{~and~} H_1 = Id: N_\infty \to N_\infty
\]
such that the diameter of the set $ \{\Phi (H_s (z,t))~|~0 \leq s \leq1\}$ goes uniformly (i.e. independent of $z        \in \partial N$) to zero, as $t \to \infty $;
\item There exists a proper continuous homotopy $\{H'_s \}_{0\leq s \leq 1}$ between
\[
H'_0 = G \circ F \text{ and } H'_1 = Id: M_\infty  \to M_\infty
\]
such that the diameter of the set $ \{\Psi (H'_s (y,t))~|~ 0 \leq s \leq1\}$ goes uniformly (i.e. independent of $y \in \partial M$) to zero, as $t \to \infty $.
\end{enumerate}
\end{definition}

The following geometric definition of $L$-groups due to Wall \cite{W1970a}.

\begin{definition}[Objects for definition of $L_n(\pi_1 X)$, \cite{WXY16}]\label{definition L}
An object
\[
\theta = (M,\partial M,\phi ,N,\partial N,\psi,f)
\]
of $L_n(\pi_1 X)$ consists the following data:
\begin{enumerate}
\item Two manifolds with boundary $M$ and $N$ both of dimension $n$;
\item Continuous maps $\phi: M \to X$ and $\psi: N \to X$;
\item A degree one normal map $f : (M,\partial M) \to (N,\partial N)$ such that $\psi \circ f = \phi$.
Moreover, on the boundary $f|_{\partial M} : \partial M \to \partial N$ is a homotopy equivalence.
\end{enumerate}	
\end{definition}

If $\theta = (M,\partial M,\phi ,N,\partial N,\psi,f)$ is an element, then we denote by $-\theta$ to be the same object except that the fundamental classes of $M$ and $N$ switch sign.
For two objects $\theta_1$ and $\theta_2 $, we write $\theta_1+\theta_2 $ to be the disjoint union of  $\theta_1$ and $\theta_2 $.
This sum operation is clearly commutative and associative, and admits a zero element: the element with $M$ (hence $N$) empty.
We denote the zero element by $0$.

\begin{definition}[Equivalence relation for definition $L_n(\pi_1 X)$, \cite{WXY16}]\label{definition eL}
Let
\[
\theta = (M,\partial M,\phi ,N,\partial N,\psi,f)
\]
be an object from $L_n(\pi_1 X)$.
We write $\theta \sim 0$ if the following conditions are satisfied:
\begin{enumerate}
\item There exists a manifold with boundary $W$ and a continuous map $\Phi: W \to X$.
The dimension of manifold $W$ is $n+1$.
Moreover, $\partial W = M \sqcup_{\partial M} \partial_2 W$, in particular $\partial M = \partial\partial _2 W$ (We can decompose $\partial W$ into two pieces $\partial_1 W$ and $\partial_2 W$, each pieces is a manifold with boundary. Moreover, $\partial W =\partial_1 W \sqcup \partial_2 W$ in particular $\partial\partial_1 W=\partial\partial_2 W$);
\item Similarly, there exists a manifold with boundary $V$ and a continuous map $\Psi: V \to X$.
The dimension of manifold $V$ is $n+1$.
Moreover, $\partial V = N \sqcup_{\partial N} \partial_2 V$, in particular $\partial N = \partial\partial_2 V$;
\item There is a degree one normal map $F : (W,\partial W) \to (V,\partial V)$ such that $\Psi \circ F = \Phi$.
Moreover, $F$ restricts to $f$ on $M$, and $F|_{\partial_2 W}: \partial_2 W \to \partial_2 V$ is a homotopy equivalence over $X$.
\end{enumerate}	
\end{definition}

A little bit of abuse our notation.
We denote by $L_n (\pi_1 X)$ the set of equivalence classes from Definition \ref{definition L} under equivalence relation from Definition \ref{definition eL}.
Note that $L_n (\pi X)$ is an abelian group under disjoint union.
It is a theorem of Wall that the above definition of $L$-groups is equivalent to the algebraic definition of $L$-groups \cite[Chapter 9]{W1970a}.

\begin{definition}[Objects for definition of $\mathcal{N}_n(X)$, \cite{WXY16}]\label{definition N}
An object
\[
\theta = (M,\partial M,\phi ,N,\partial N,\psi,f)
\]
of $\mathcal{N}_n(X)$ consists of the following data:
\begin{enumerate}
\item Two manifolds with boundary $M$ and $N$ both of dimensional $n$;
\item Continuous maps $\phi: M \to X$ and $\psi: N \to X$;
\item A degree one normal map $f : (M,\partial M) \to (N,\partial N)$ such that $\psi  \circ f = \phi$.
Moreover, on the boundary $f|_{\partial M} : \partial M \to \partial N$ is an \emph{infinitesimally controlled homotopy equivalence} over $X$.
\end{enumerate}	
\end{definition}

\begin{definition}[Equivalence relation for definition of $\mathcal{N}_n(X)$, \cite{WXY16}]\label{definition eN}
Let
\[
\theta = (M,\partial M,\phi ,N,\partial N,\psi,f)
\]
be an object from $\mathcal{N}_n(X)$.
We write $\theta   \sim 0$ if the following conditions are satisfied.
\begin{enumerate}
\item There exists a manifold with boundary $W$ and a continuous map $\Phi: W \to X$.
The dimension of manifold $W$ is $n+1$.
Moreover, $\partial W = M \sqcup_{\partial M} \partial_2 W$, in particular $\partial M = \partial partial _2 W$;
\item Similarly, there exists a manifold with boundary $V$ and a continuous map $\Psi: V \to X$.
The dimension of manifold $V$ is $n+1$.
Moreover, $\partial V = N \sqcup_{\partial N} \partial_2 V$, in particular $\partial N = \partial\partial_2 V$;
\item There is a degree one normal map $F : (W,\partial W) \to (V,\partial V)$ such that $\Psi \circ F = \Phi$.
Moreover, $F$ restricts to $f$ on $M$, and $F|_{\partial_2 W}: \partial_2 W \to \partial_2 V$ is an \emph{infinitesimally controlled homotopy equivalence} over $X$.
\end{enumerate}	
\end{definition}

We denote by $\mathcal{N}_n(X)$ the set of equivalence classes from Definition \ref{definition N} under equivalence relation from Definition \ref{definition eN}.
Note that $\mathcal{N}_n(X)$ is an abelian group with the sum operation being disjoint union.
$\mathcal{N}_n(X)$ is the usual Normal group in surgery theory.
Now we recall the relative $L$-group $L_n(\pi_1 X,X)$.

\begin{definition}[Objects for definition of $L_n(\pi_1 X,X)$, \cite{WXY16}]\label{definition Ln}
An object
\[
\theta = (M,\partial_{\pm}M,\phi ,N,\partial_{\pm} N,\psi,f)
\]
of $L_n(\pi_1X,X)$ consists of the following data:
\begin{enumerate}
\item Two manifolds with boundary $M$ and $N$ both of dimensional $n$.
Moreover, $\partial M =\partial_+ M \sqcup \partial_- M$(resp. $\partial N =\partial_{+} N \sqcup \partial_{-} N$) where $\partial M $(resp. $\partial N$) is the boundary of $M$(resp. $N$).
In particular, $\partial\partial_+ M = \partial\partial_- M$ and $\partial\partial_ + N = \partial\partial_- N$;
\item Continuous maps $\phi: M \to X$ and $\psi: N \to X$;
\item  A degree one normal map $f : (M,\partial M) \to (N,\partial N)$ such that $\psi \circ f = \phi$;
\item  The restriction $f|_{ \partial_{+} M }: \partial_{+} M \to \partial_+ N$ is a homotopy equivalence over $X$;
\item  The restriction $f|_{ \partial_{-} M }: \partial_{-} M \to \partial_- N$ is a degree one normal map over $X$;
\item  The homotopy equivalence $f|_{\partial_{-} M}$ restricts to an \emph{infinitesimally controlled homotopy equivalence} $f|_{\partial_{\pm} M }: \partial_{\pm} M \to \partial_{\pm} N$ over $X$.
\end{enumerate}	
\end{definition}

\begin{definition}[Equivalence relation for definition of $L_n(\pi_1X,X)$, \cite{WXY16}]\label{definition eLn}
Let
\[
\theta = (M,\partial_{\pm} M,\phi ,N,\partial_{\pm} N,\psi,f)
\]
be an object from $L_n(\pi_1X,X)$.
We write $\theta \sim 0$ if the following conditions are satisfied.
\begin{enumerate}
\item There exists a manifold with boundary $W$ and a continuous map $\Phi: W \to X$.
The dimension of manifold $W$ is $n+1$.
Moreover, $\partial W = M  \sqcup  \partial_2 W \sqcup \partial_3 W$ and we have decompositions $\partial M= \partial_{+} M \sqcup \partial_{-} M$, $\partial\partial_2 W = \partial\partial_{2,+}W \sqcup  \partial\partial_{2,-}W$ and $\partial\partial_3 W = \partial\partial_{3,+}W  \sqcup  \partial\partial_{3,-}W$ such that
\[
\partial_+M = \partial\partial_{2,+}W,\;\partial_-M = \partial\partial_{3,-}W,\; \text{ and }\;\partial\partial_{2,-}W = \partial\partial_{3,+}W.
\]
Furthermore, we have
\[
\partial_+M \cap \partial_-M = \partial\partial_{2,+}W \cap \partial\partial_{2,-}W = \partial\partial_{3,+}W  \cap \partial\partial_{3,-}W;
\]
\item Similarly, there exists a manifold with boundary $V$ and a continuous map $\Psi: V \to X$.
The dimension of manifold $V$ is $n+1$.
Moreover, $\partial V = N \sqcup \partial_2 V \sqcup \partial_3 V$ satisfies similar conditions as $W$;
\item There is a degree one normal map $F : (V,\partial V) \to (W,\partial W)$ such that $\Psi \circ F = \Phi$.
Moreover, $F$ restricts to $f$ on $N$;
\item $F|_{\partial_2 W} : \partial_2 W\to  \partial_2 V$ is a homotopy equivalence over $X$;
\item $F|_{\partial_2 W}$ restricts to an \emph{infinitesimally controlled homotopy equivalence} $F| _{\partial\partial_{2,-} W} : \partial\partial_{2,-} W  \to \partial\partial_{2,-} V$ over $X$.
\end{enumerate}	
\end{definition}

We denote by $L_n(\pi_1X,X)$ the set of equivalence classes from Definition \ref{definition Ln} under equivalence relation from Definition \ref{definition eLn}.
Note that $L_n(\pi_1X,X)$ is an abelian group with the sum operation being disjoint union.

\begin{remark}
The equivalence relation can also be stated as follows,
there exists $\theta'=(M',\partial_\pm M',\phi', N',\partial_\pm N', \psi', f') \in L_n(\pi_1X,X)$, such that
\begin{enumerate}
\item $f'$ is a homotopy equivalence ove $X$;
\item $\partial_- M'=\partial_+M$,\;$\partial_- N'=\partial_+N$,\;$f'|_{\partial_+ M'}=f|_{\partial_+ M}$,\;$\psi'|_{\partial_+ N'}=\psi|_{\partial_+ N}$;
\item $f'_{\partial_-M}$ restricts to an \emph{infinitesimally controlled homotopy equivalence} $f'_{\partial_-M}:\partial_-M\to\partial_-N$ over $X$.
\end{enumerate}
In fact, one can take $M\sqcup_{\partial_+M} M'$ as $\partial_3 W$, and $N\sqcup_{\partial_+N} N'$ as $\partial_3 V$ in definition \ref{definition eLn}.

Note that  both
\[
(M\sqcup_{\partial_+M} M')\sqcup_{\partial_-M\sqcup \partial_+M'} (M \sqcup_{\partial_+M}M')
\]
and
\[
(N\sqcup_{\partial_+N} N')\sqcup_{\partial_-N\sqcup \partial_+N'} (N \sqcup_{\partial_+N}N')
\]
are boundary of some manifold with boundary.
\end{remark}

\begin{definition}[Objects for definition of $\mathcal{S}_n(X)$, \cite{WXY16}]\label{definition S}
An object
\[
\theta = (M,\partial M,\phi ,N,\partial N,\psi,f)
\]
of $\mathcal{S}_n(X)$ consists of the following data:
\begin{enumerate}
\item Two manifolds with boundary $M$ and $N$ both of dimensional $n$;
\item Continuous maps $\phi: M \to X$ and $\psi: N \to X$;
\item Homotopy equivalence $f : (M,\partial M) \to (N,\partial N)$ such that $\psi  \circ f = \phi$.
Moreover, on the boundary $f|_{\partial M} : \partial M \to \partial N$ is an \emph{infinitesimally controlled homotopy equivalence} over $X$.
\end{enumerate}	
\end{definition}
\begin{definition}[Equivalence relation for definition of $\mathcal{S}_n(X)$, \cite{WXY16} ]\label{definition eS}
Let
\[
\theta = (M,\partial M,\phi ,N,\partial N,\psi,f)
\]
be an object from $\mathcal{S}_n(X)$.
We write $\theta        \sim 0$ if the following conditions are satisfied.
\begin{enumerate}
\item There exists a manifold with boundary $W$ and a continuous map $\Phi: W \to X$.
The dimension of manifold $W$ is $n+1$.
Moreover, $\partial W = M \sqcup_{\partial M} \partial_2 W$, in particular $\partial M = \partial\partial _2 W$;
\item Similarly, there exists a manifold with boundary $V$ and a continuous map $\Psi: V \to X$.
The dimension of manifold $V$ is $n+1$.
Moreover, $\partial V = N \sqcup_{\partial N} \partial_2 V$, in particular $\partial N = \partial \partial_2 V$;
\item There is a homotopy equivalence $F : (W,\partial W) \to (V,\partial V)$ such that $\Psi \circ F = \Phi$.
Moreover, on the boundary $f|_{\partial M} : \partial M \to \partial N$ is an \emph{infinitesimally controlled homotopy equivalence} over $X$.
\end{enumerate}	
\end{definition}

We denote by $\mathcal{S}_n(X)$ the set of equivalence classes from Definition \ref{definition S} under equivalence relation from Definition \ref{definition eS}.
Note that $\mathcal{S}_n(X)$ is an abelian group with the sum operation being disjoint union.

Following \cite{WXY16}, we begin to give a description of the surgery long exact sequence based on ideas of Wall.
There is natural group homomorphism
\[
i _* : \mathcal{N}_n (X) \to  L_n (\pi_1 X)
\]
by forgetting control.
Moreover, every element
\[
\theta = (M, \partial M, \phi , N, \partial N,\psi , f) \in  L_n (\pi_1 X)
\]
naturally defines an element in $L_n (\pi_1 X,X)$ by letting $\partial_-  M =\emptyset $.

We denote the corresponding natural homomorphism by
\[
j_* :L_n (\pi_1 X)\to L_n (\pi_1 X,X).
\]
Moreover, for $\theta = (M, \partial_{\pm} M, \phi , N, \partial_{\pm}  N,\psi , f)\in  L_n (\pi_1 X,X)$, we see that
$\theta_- = (\partial_-M, \partial\partial_- M, \phi , \partial_-N, \partial\partial_- N,\psi , f)$ defines an element in $\mathcal{N}_{n-1} (X)$.

We denote the corresponding natural homomorphism by
\[
\partial_* :L_n (\pi_1 X,X) \to \mathcal{N}_{n-1} (X).
\]

\begin{theorem}[{\cite[Theorem 3.22]{WXY16}}]\label{theorem LLNexact}
We have the following long exact sequence
\[
\dots\to  \mathcal{N}_n(X) \stackrel{i_*}{\to} L_n(\pi_1 X) \stackrel{j_*}{\to} L_n(\pi_1 X,X)\stackrel{\partial_*}{\to} \mathcal{N}_{n-1}(X)\to \dots.
\]
\end{theorem}

We shall identify the groups $\mathcal{S}_n(X)$ and $L_{n+1} (\pi_ 1 X, X)$.
There is a natural group homomorphism
\[
c_* : \mathcal{S}_n(X)\to L_{n+1} (\pi_ 1 X, X)
\]
by mapping
\[
\theta = (M,\partial M,\phi ,N,\partial N,\psi,f) \mapsto \theta \times I
\]
where $\theta \times I$ consists of the following data:
\begin{enumerate}
  \item Manifold $(M \times I, \partial_{\pm} (M \times I))$ with $\partial_+(M \times I) = M\times \{0\}$ and $ \partial_-(M \times I)= \partial M\times I\sqcup M\times \{1\}$.
In particular, we have $\partial\partial_+(M \times I)= \partial M = \partial\partial_-(M \times I)$.
Similarly, for $N$ we have manifold $(N \times I, \partial_{\pm} (N \times I))$;

  \item Continuous maps $\tilde{\phi} := \phi\circ p : M\times I \stackrel{p}{\to} M\stackrel{\phi}{\to} X$ and $\tilde{\psi} := \psi\circ p : N \times I \stackrel{p}{\to} N\stackrel{\psi}{\to} X$, where $p$ is the canonical projection map from $M\times I \to M$ or from $N\times I \to N$;
  \item A degree one normal map $\tilde{f} := f \times Id: (M \times I, \partial_{\pm} (M \times I))\to(N \times I, \partial_{\pm} (N \times I))$ such that $\tilde{\psi}\circ \tilde{f} =\tilde{\phi}$;
  \item The restriction $\tilde{f}|_{\partial_+(M \times I)}: \partial_+(M \times I) \to \partial_+(N \times I)$ is a homotopy equivalence over $X$;
  \item The restriction $\tilde{f}|_{\partial_-(M \times I)}: \partial_-(M \times I) \to \partial_-(N \times I)$ is a degree one normal map over $X$;
  \item The homotopy equivalence $\tilde{f}|_{\partial_+(M \times I)}: \partial_+(M \times I) \to \partial_+(N \times I)$ over $X$ restricts to an \emph{infinitesimally controlled homotopy equivalence} $\tilde{f}|_{\partial\partial_+(M \times I)}: \partial\partial_+(M \times I) \to \partial\partial_+(N \times I)$ over $X$.
\end{enumerate}

For 
\[\theta = (M, \partial_{\pm} M, \phi , N, \partial_{\pm}  N,\psi , f)\in  L_{n+1} (\pi_1 X,X),\]
we see that 
\[\theta_+ = (\partial_+M, \partial\partial_+ M, \phi , \partial_+N, \partial\partial_+ N,\psi , f)\] defines an element in $\mathcal{S}(X)$.

We denote the natural group homomorphism by
\[
r_* : L_{n+1} (\pi_1 X, X) \to  \mathcal{S}_n(X).
\]

It follows from definition that the homomorphisms $c_*$ and $r_*$ are well defined.
By \cite[Proposition 3.23]{WXY16} the homomorphisms $c_*$ and $r_*$ are inverses of each other.
In particular, we have $\mathcal{S}_n(X) \cong L_{n+1} (\pi_1 X, X)$. Hence, we have the following theorem.
\begin{theorem}[\cite{WXY16}]\label{theorem surgery exact}
We have the following long exact sequence:
\[
\dots\to  \mathcal{N}_{n+1}(X)\to L_{n+1}(\pi_1 X)\to \mathcal{S}_n(X)\to \mathcal{N}_n(X)\to \dots.
\]
\end{theorem}

\begin{remark}
When $X$ is of dimension $n\geq 5$, then by \cite[Proposition 3.31]{WXY16} $ \mathcal{S}_n(X)$ is isomorphic to $\mathcal{S}^{TOP}(X)$.
\end{remark}

Our paper is based on the Proposition \ref{appendix1}.
The point here is that in the definition of $\mathcal{S}_n^{C^{\infty}}(X)$ we use smooth manifolds with boundary.
Hence, our constructions in previous sections works.
We first introduce $\mathcal{S}_n^{C^{\infty}}(X)$.

\begin{definition}[{\cite[Definition 3.35]{WXY16}}]
An element $\theta \in \mathcal{S}_n^{C^{\infty}}(X)$ consists of the following data:
\begin{enumerate}
  \item $\theta$ is an element of $\mathcal{S}_n(X)$;
  \item $M$ and $N$ are smooth manifolds with boundary, and the map $f : M \to N$ is smooth.
\end{enumerate}
\end{definition}

\begin{proposition}[{\cite[Proposition 3.36]{WXY16}}]\label{appendix1}
For $n \geq 5$, we have natural isomorphisms
\[
\mathcal{S}_n^{C^{\infty}}(X) \cong \mathcal{S}_n(X).
\]
In particular, every element in $\mathcal{S}_n(X)$ has a smooth representative.
\end{proposition}

\section{Poincar\'e duality operator}\label{section Poin duality oper}
In this appendix we recall the concept of \emph{Hilbert–Poincar\'e complex}.

We shall begin with \emph{complex of Hilbert spaces}:
\[
\begin{tikzcd}
\mathcal{H}_0\arrow{r}{d_1}&\mathcal{H}_1\arrow{r}{d_2}& \cdots \arrow{r}{d_n}&\mathcal{H}_n
\end{tikzcd}
\]
where $\mathcal{H}_i$ are Hilbert spaces, and $d_i$ are closable, unbounded operators.
We shall denote by $\mathcal{H}:=\oplus \mathcal{H}_i$ the direct sum of $\mathcal{H}_i$, denote by $d:=\oplus d_i$ the direct sum of $d_i$ acting on $\mathcal{H}$.
It is a closable operator, and $d^2 =0$.

\begin{definition}\label{poincaredualityoperator}
An $n$-dimensional \emph{Hilbert-Poincar\'e complex} is a complex of Hilbert spaces
\[
\begin{tikzcd}
\mathcal{H}_0\arrow{r}{d_1}&\mathcal{H}_1\arrow{r}{d_2}& \cdots \arrow{r}{d_n}&\mathcal{H}_n
\end{tikzcd}
\]
together with a bounded operator $S:\mathcal{H}\to\mathcal{H}$ such that
\begin{enumerate}
  \item $S$ is self-adjoint, and $S:\mathcal{H}_p\to\mathcal{H}_{n-p}$;
  \item $S$ can be approximated by bounded operators with finite propagation; 
  \item $S$ map the domain of $d$ into the domain of $d^*$, and $Sd+d^*S=0$;
  \item $S$ is a chain homotopy equivalence from the dual complex
\[
\begin{tikzcd}
\mathcal{H}_n\arrow{r}{d^*_n}&\mathcal{H}_{n-1}\arrow{r}{d^*_{n-1}}& \cdots \arrow{r}{d^*_1}&\mathcal{H}_0
\end{tikzcd}
\]
to the homology of the complex $(\mathcal{H},d)$.
\item there exists chain homotopy inverse of $S$, $\overline{S}$, which can be approximated by finite propagation operator, such that 
there exist $h$, $\overline{h}$, both of which can be approximated by finite propagation operator, and 
\[
I-S\overline{S}= d^*h -hd^*, \overline{S}S-I= d\overline{h}-\overline{h}d.
\] 
\end{enumerate}
\end{definition}

We shall call the operator $S$ a \emph{Poincar\'e duality operator} of $(\mathcal{H},d)$ since it induces Poincar\'e duality on the homology of $(\mathcal{H},d)$.
Note that, in order to eliminate most of the $\pm 1$ signs, our definition of Hilbert–Poincar\'e complex and Poincar\'e duality operator slightly different from the one in \cite[Definition 3.1]{HR1}.
If we denote $d+d^*$ by $D$, we know that $D\pm S$ are invertible (\cite[Lemma 3.5]{HR1}, \cite[section 5]{HR1}).

\begin{example}\label{ex.de Rham complex}
Let $M$ be an oriented, closed smooth (complete Riemannian) manifold of dimension $n$ and let $ \tilde{M}$ be a regular $G$-covering space of $M$, and $G$ is a discrete group.
We have \emph{de Rham complex}
\[
\begin{tikzcd}
\Lambda^0(\tilde{M})\arrow{r}{d_{\tilde{M}}}&\Lambda^1(M)\arrow{r}{d_{\tilde{M}}}& \cdots \arrow{r}{d_{\tilde{M}}}&\Lambda^n(\tilde{M}),
\end{tikzcd}
\]
of compact support differential forms on $\tilde{M}$.
Passing to $L^2$-completions, we obtain a complex of Hilbert spaces
\[
\begin{tikzcd}
\Lambda^0_{L^2}(\tilde{M})\arrow{r}{d_{\tilde{M}}}&\Lambda^1_{L^2}(\tilde{M})\arrow{r}{d_{\tilde{M}}}& \cdots \arrow{r}{d_{\tilde{M}}}&\Lambda^n_{L^2}(\tilde{M}).
\end{tikzcd}	
\]
Since it is convenient to work with \emph{closed} unbounded operators, here we shall denote by $d_{\tilde{M}}$ the \emph{operator-closure} of the de Rham differential on $\Lambda^*_{L^2}(\tilde{M})$.

Let $*$ be the Hodge-$*$ operator.
Define
\[
S_{\tilde{M}}(\omega)=i^{p(p-1)+[\frac{n}{2}]}*\omega,\;\omega\in \Lambda^p(\tilde{M}),
\]
which is a self-adjoint operator with $S_{\tilde{M}}^2=1$, anti-commutes with the operator $d_{\tilde{M}} + d^*_{\tilde{M}}$.
And $S_{\tilde{M}}$ a \emph{Poincar\'e duality operator} of $(\bigoplus \Lambda^p_{L^2}(\tilde{M}),d)$.
\end{example}

\begin{example}
Let $X$ be a closed oriented topological manifold of dimension $n$ and let $\tilde{X}$ be a regular $G$-covering of $X$, and $G$ is a discrete group.
Moreover, we have the following maps
\[
\begin{tikzcd}[column sep=small]
M  \arrow[rr, "f"] \arrow[dr,"\phi"'] & & N  \arrow[dl,"\psi"]\\
& X  &
\end{tikzcd}	
\]
where, $f:M \to N$ is a smooth orientation-preserving homotopy equivalence between two closed manifolds $M$ and $N$, $\phi$ and $\psi$ are continuous maps, such that $\phi=\psi\circ f$.
Let $\tilde{N}=\psi^*(\tilde{X})$ and $\tilde{M}=f^*(\tilde{N})=\phi^*(\tilde{X})$ be the pull back coverings.
We have the following complex of Hilbert spaces
\medskip

\begin{adjustbox}{center}
\begin{tikzcd}
\Lambda^0_{L^2}(\tilde{M})\oplus \Lambda^0_{L^2}(\tilde{N})\arrow{r}{d}&\Lambda^1_{L^2}(\tilde{M})\oplus\Lambda^1_{L^2}(\tilde{N}) \arrow{r}{d}& \cdots \arrow{r}{d}&\Lambda^n_{L^2}(\tilde{M})\oplus\Lambda^n_{L^2}(\tilde{N}),
\end{tikzcd}	
\end{adjustbox}
where, $d :=\begin{pmatrix}
d_{\tilde{M}} & 0\\
0 & d_{\tilde{N}}
\end{pmatrix}$.

We denote $S :=\begin{pmatrix}
S_{\tilde{M}} & 0\\
0 & -S_{\tilde{N}}
\end{pmatrix}$.
Then $S$ is a \emph{Poincar\'e duality operator} of $(L^2(\Lambda(\tilde{M})) \bigoplus L^2(\Lambda(\tilde{N})),d)$.
Moreover, by item 5 at the end of subsection \ref{ssec.HSsubmersion} and Lemma \ref{lemma homo poincare duality topic}, for $t \in [0,1]$, operators
\begin{eqnarray*}
&&\begin{pmatrix}
S_{\tilde{M}} & 0\\
0 & -(1-t)S_{\tilde{N}}-t T_{f}^*S_{\tilde{M}} T_{f}\end{pmatrix}, \begin{pmatrix}
\cos( t\frac{\pi}{2})S_{\tilde{M}} & \sin(t\frac{\pi}{2}) S_{\tilde{M}} T_{f}\\
\sin(t\frac{\pi}{2}) T_{f}^*S_{\tilde{M}} & -\cos(t\frac{\pi}{2}) T_{f}^* S_{\tilde{M}} T_{f}\end{pmatrix},\\
&&\begin{pmatrix}
\cos( t\frac{\pi}{2})S_{\tilde{M}} & \sin(t\frac{\pi}{2}) S_{\tilde{M}} T_{f}\\
\sin(t\frac{\pi}{2}) T_{f}^*S_{\tilde{M}} & -\cos(t\frac{\pi}{2}) T_{f}^* S_{\tilde{M}} T_{f}
\end{pmatrix},  \begin{pmatrix}
0& e^{it \pi}S_{\tilde{M}} T_{f}\\
e^{-it\pi} T_{f}^*S_{\tilde{M}} & 0
\end{pmatrix}
\end{eqnarray*}
are all \emph{Poincar\'e duality operators} of $(L^2(\Lambda(\tilde{M})) \bigoplus L^2(\Lambda(\tilde{N})),d)$.
\end{example}

\begin{remark}
If $X, M, N$ are not assumed to be compact, then ``homotopy equivalece'' $f$ needs to be replaced by controlled one instead.
\end{remark}

\section{Proof of product formula}\label{subsection X times R proof}
Now, we will prove the even case of the product formula in subsection \ref{subsection X times R}.
If $\theta=(M,\partial M,\phi, N, \partial N,\psi, f) \in L_n(\pi_1 X)$, let $ \theta \times  \mathbb{R}$ be the product of $\theta $ and $\mathbb{R}$, which defines an element in $L_{n+1}(\pi_1 X)$.
Here various undefined terms take the obvious meanings \cite{WXY16}.

Note that the construction in subsection \ref{subsection M infty and N infty} also applies to $\theta \times  \mathbb{R}$ and defines a $K$-theory class in $K_{n+1}(C^*(\tilde{X}\times  \mathbb{R} )^G)$.
And there is a natural homomorphism $\alpha : C^*(\tilde{X} )^G \otimes  C^*_L(\mathbb{R}) \to C^*(\tilde{X}\times  \mathbb{R} )^G$.

\begin{proposition}\label{proposition times proposition proof}
With the same notation as above, we have
\[
k_n \cdot \alpha_*\left(\text{Ind}(\theta)\otimes \text{Ind}_L(\mathbb{R})\right)=\text{Ind}(\theta \times  \mathbb{R})
\]
in $K_{n+1}(C^*(\tilde{X}\times \mathbb{R} )^G)$, where $\text{Ind}_L(\mathbb{R})$ is the $K$-homology class of the signature operator on $\mathbb{R}$ and $k_n=1$ if $n$ is even and $2$ if $n$ is odd.
\end{proposition}


Let $Z$ and $Z'$ be (smooth) manifolds of dimensional $n$ and $n'$ respectively.
Then $L^2(\Lambda (Z\times Z'))=L^2(\Lambda (Z))\otimes L^2(\Lambda (Z'))$.
$d_{Z\times Z'}=d_{Z} \hat{\otimes}1+1\hat{\otimes} d_{Z'}$, where $\hat{\otimes}$ stands for graded tensor product.
And $S_{Z\times Z'}(x\otimes y)=(-1)^{(n-p)q}S_{Z}(x)\otimes S_{Z'}(y)$, where $x\otimes y\in \Lambda^p(Z)\otimes\Lambda^q(Z')$.

\begin{prf}
With the notations as above, and borrow the notations from subsection \ref{subsection M infty and N infty}.

We use $\Lambda$ denote $\Lambda(\tilde{M}_\infty)\oplus\Lambda(\tilde{N}_\infty)$, and $\Lambda_{\times\mathbb{R}}$ denote $\Lambda(\tilde{M}_\infty\times\mathbb{R})\oplus\Lambda(\tilde{N}_\infty\times\mathbb{R})$.

By $h dt\mapsto h$ we can identify $\Lambda^{1}(\mathbb{R})$ with $\Lambda^{0}(\mathbb{R})$.
With this identification, we have $d_{\mathbb{R}}^* = -d_{\mathbb{R}}$ and $d_{\mathbb{R}}$ is skew-adjoint.

Then we have
\begin{align*}
\Lambda_{\times\mathbb{R}}^{odd}=&(\Lambda^{even}\otimes\Lambda^{1}(\mathbb{R}))\oplus(\Lambda^{odd}\otimes\Lambda^{0}(\mathbb{R}))\\
\cong & (\Lambda^{even}\oplus\Lambda^{odd})\otimes\Lambda^{0}(\mathbb{R}) = \Lambda\otimes \Lambda^{0}(\mathbb{R});\\
\Lambda_{\times\mathbb{R}}^{even}=&(\Lambda^{even}\otimes\Lambda^{0}(\mathbb{R}))\oplus(\Lambda^{odd}\otimes\Lambda^{1}(\mathbb{R}))\\
\cong &(\Lambda^{even}\oplus\Lambda^{odd})\otimes\Lambda^{0}(\mathbb{R})= \Lambda\otimes \Lambda^{0}(\mathbb{R}).	
\end{align*}

If we denote
\[
d_{\infty}:=\begin{pmatrix}
d_{\tilde{M}_\infty}&0\\
0 &d_{\tilde{N}_\infty}
\end{pmatrix}\text{~,~}
\tilde{d}_{\infty}:=\begin{pmatrix}
d_{\tilde{M}_\infty\times  \mathbb{R}}&0\\
0 &d_{\tilde{N}_\infty\times  \mathbb{R}}
\end{pmatrix}
\]
and $D_{\mathbb{R}}:=id_{\mathbb{R}}$.

Then as a map from $\Lambda\otimes \Lambda^{0}(\mathbb{R})$ to $\Lambda\otimes \Lambda^{0}(\mathbb{R})$.
We have
\begin{gather*}
\begin{aligned}
\tilde{d}_{\infty}+\tilde{d}_{\infty}^*&=(d_{\infty}+d_{\infty}^*)\otimes1-1\otimes i D_{\mathbb{R}};	\\
S_{(\tilde{M}_\infty\sqcup\tilde{N}_\infty) \times  \mathbb{R}}&=S_{\tilde{M}_\infty\sqcup\tilde{N}_\infty}\otimes 1.	
\end{aligned}
\end{gather*}

We now define the ``\emph{Poincar\'e duality operators}".
As the equation above we can take the ``\emph{Poincar\'e duality operators}" as $S_f\otimes1$ and $S'_f\otimes1$.
And obviously, everything involved here is \emph{controlled} over $\tilde{X}\times \mathbb{R}$.
By Proposition \ref{proposition almost invertible}.
We know that $\text{Ind}(\theta \times  \mathbb{R})$ can be represented by invertible operator
\[
(D\otimes1-1\otimes iD_{\mathbb{R}}+\alpha S_f\otimes 1)(D\otimes1-1\otimes iD_{\mathbb{R}}-\alpha S'_f\otimes 1)^{-1}
\]
in $B(L^2(\Lambda_{\times\mathbb{R}}^{even}))$.

Recall that $(D+\alpha S_f)$ and $(D-\alpha S'_f)$ is a self-adjoint invertible operator.
Therefore, $(D+\alpha S_f)$ and $(D-\alpha S'_f)$ is homotopic to $P_{+}(D+\alpha S_f)-P_{-}(D+\alpha S_f)$ and  $P_{+}(D-\alpha S'_f)-P_{-}(D-\alpha S'_f)$ through a path of invertible elements, where $P_{\pm}(D+\alpha S_f)$ and  $P_{\pm}(D-\alpha S'_f)$ is the positive/negative projection of $(D+\alpha S_f)$ and $(D-\alpha S'_f)$.

For notational simplicity, we use $P_{\pm}$ and $P'_{\pm}$ denote $P_{\pm}(D+\alpha S_f)$ and  $P_{\pm}(D-\alpha S'_f)$ respectively.

We see that
\[
D\otimes1-1\otimes iD_{\mathbb{R}}+\alpha S_f\otimes 1=(D+\alpha S_f)\otimes 1-1\otimes iD_{\mathbb{R}}
\]
is homotopic to
\[
(P_+-P_-)\otimes1-(P_++P_-)\otimes iD_{\mathbb{R}}=P_+\otimes(1-iD_{\mathbb{R}})-P_-\otimes(1+iD_{\mathbb{R}}).
\]
Similarly,
\[
D\otimes1-1\otimes iD_{\mathbb{R}}-\alpha S'_f\otimes 1=(D-\alpha S'_f)\otimes 1-1\otimes iD_{\mathbb{R}}
\]
is homotopic to
\[
(P'_+-P'_-)\otimes1-(P'_++P'_-)\otimes iD_{\mathbb{R}}=P'_+\otimes(1-iD_{\mathbb{R}})-P'_-\otimes(1+iD_{\mathbb{R}}).
\]
A routine calculation shows that
\[
(P'_+\otimes(1-iD_{\mathbb{R}})-P'_-\otimes(1+iD_{\mathbb{R}}))^{-1}=P'_+\otimes(1-iD_{\mathbb{R}})^{-1}-P'_-\otimes(1+iD_{\mathbb{R}})^{-1}.
\]
It follows that,
\begin{align*}
&[(P_+\otimes(1-iD_{\mathbb{R}})-P_-\otimes(1+iD_{\mathbb{R}}))(P'_+\otimes(1-iD_{\mathbb{R}})^{-1}-P'_-\otimes(1+iD_{\mathbb{R}})^{-1}]\\
=&\begin{aligned}[t][P_+P'_+\otimes1-P_+ P'_-\otimes(1-iD_{\mathbb{R}})(1+iD_{\mathbb{R}})^{-1}\\
 {}+P_-P'_+\otimes(1+iD_{\mathbb{R}})(iD_{\mathbb{R}}-1)^{-1}\\
 {}+P_-P'_-\otimes1] 
\end{aligned}\\
=&\begin{aligned}[t][(1-P_+)\otimes1)(P'_+\otimes(1+iD_{\mathbb{R}} )(iD_{\mathbb{R}}-1)^{-1}+(1-P'_+)\otimes1)\\
{}+(P_+\otimes(iD_{\mathbb{R}}-1)(1+iD_{\mathbb{R}})^{-1}]
\end{aligned}\\
=&([P_+]-[P'_+])\otimes[(D_{\mathbb{R}}+i)(D_{\mathbb{R}}-i)^{-1}]
\end{align*}

the last term is precisely $\text{Ind}(\theta)\otimes\text{Ind}_L(\mathbb{R})$.

To summarize, when $n$ is even, we have proved that
\[
\alpha_*\left(\text{Ind}(\theta)\otimes \text{Ind}_L(\mathbb{R})\right)=\text{Ind}(\theta \times  \mathbb{R}).
\]	
\end{prf}

\bibliography{additive}

\end{document}